\documentclass[12pt,a4paper,reqno]{amsart}
\usepackage[marginratio=1:1,totalwidth=15.75cm,totalheight=22.275cm]{geometry}
\usepackage{amsmath,amsfonts,amsthm,amssymb,graphicx}

\usepackage[shortlabels,inline]{enumitem}

\usepackage{transparent}
\usepackage[dvipsnames]{xcolor}
\usepackage{hyperref}

\usepackage[flushmargin]{footmisc}

\newcommand*{\defeq}{\mathrel{\vcenter{\baselineskip0.5ex \lineskiplimit0pt
                     \hbox{\scriptsize.}\hbox{\scriptsize.}}}%
                     =}

\newcommand{\deriv}{\mathrm{d}}

\numberwithin{equation}{section}

\theoremstyle{plain}
\newtheorem{thm}{Theorem}[section]
\newtheorem{lem}[thm]{Lemma}
\newtheorem{prop}[thm]{Proposition}
\newtheorem{cor}[thm]{Corollary}
\newtheorem{question}[thm]{Question}
\newtheorem{conjecture}[thm]{Conjecture}

\theoremstyle{definition}
\newtheorem{dfn}[thm]{Definition}

\theoremstyle{remark}
\newtheorem{rmk}[thm]{Remark}

\newcommand{\im}{\operatorname{Im}}
\newcommand{\re}{\operatorname{Re}}
\newcommand{\BU}{BU}
\newcommand{\mav}{\operatorname{mav}}
\newcommand{\I}{I}
\newcommand{\K}{K}
\newcommand{\J}{J}
\newcommand{\F}{F}

\newcommand{\D}{D}

\newcommand{\diam}{\operatorname{diam}}

\renewcommand{\theta}{\vartheta}
\renewcommand{\phi}{\varphi}

\newcommand{\eps}{\varepsilon}
\newcommand{\interior}{\operatorname{int}}

\newcommand{\DD}{\mathbb{D}}
\newcommand{\HH}{\mathbb{H}}
\newcommand{\dist}{\operatorname{dist}}
\newcommand{\Log}{\operatorname{Log}}

\newcommand{\Q}{\mathbb{Q}}

\newcommand \C{\mathbb{C}}
\newcommand \CR{\widehat{\mathbb{C}}}
\newcommand \Ch{\widehat{\mathbb{C}}}
\newcommand \N{\mathbb{N}}
\newcommand \R{\mathbb{R}}

\newcommand{\B}{\mathcal{B}}

\newcommand{\Fill}{\operatorname{fill}}

\makeatletter
\renewcommand*{\eqref}[1]{%
  \hyperref[{#1}]{\textup{\tagform@{\ref*{#1}}}}%
}
\makeatother

\begin{document}

\title[Eremenko's conjecture and wandering Lakes of Wada]{Eremenko's conjecture, \\ wandering Lakes of Wada, \\ and maverick points}
\author[{D. Mart\'i-Pete \and L. Rempe \and J. Waterman}]{David Mart\'i-Pete \and Lasse Rempe \and James Waterman}

\address{Department of Mathematical Sciences\\ University of Liverpool\\ Liverpool L69 7ZL\\ United Kingdom\textsc{\newline \indent \href{https://orcid.org/0000-0002-0541-8364}{\includegraphics[width=1em,height=1em]{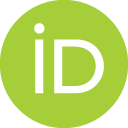} {\normalfont https://orcid.org/0000-0002-0541-8364}}}} 
\email{david.marti-pete@liverpool.ac.uk}

\address{Department of Mathematics\\ University of Manchester\\ Manchester M13 9PL\\ United Kingdom\textsc{\newline \indent \href{https://orcid.org/0000-0001-8032-8580}{\includegraphics[width=1em,height=1em]{orcid2.png} {\normalfont https://orcid.org/0000-0001-8032-8580}}}} 
\email{lasse.rempe@manchester.ac.uk}

\address{Institute for Mathematical Sciences\\ Stony Brook University\\ Stony Brook NY 11794\\ USA\textsc{\newline \indent \href{https://orcid.org/0000-0001-7266-0292}{\includegraphics[width=1em,height=1em]{orcid2.png} {\normalfont https://orcid.org/0000-0001-7266-0292}}}}
\email{james.waterman@stonybrook.edu}

\subjclass[2020]{Primary 37F10; Secondary 30D05, 37B45, 54F15.}
\keywords{complex dynamics, Julia set, Eremenko's conjecture, escaping set, Lakes of Wada, wandering domain, maverick points}

\date{\today}

\dedicatory{Dedicated to Alex Eremenko and Misha Lyubich.}

\begin{abstract}
  We develop a general technique for realising full closed subsets of the complex plane
    as wandering sets of entire functions. Using this construction, we solve
    a number of open problems.
    \begin{enumerate}[(1)]
      \item We construct a counterexample to \emph{Eremenko's conjecture}, a central 
         problem in transcendental dynamics that asks whether
         every connected component of the set of escaping points of a transcendental entire 
         function is unbounded. 
    \item 
    We prove that there is 
      a transcendental entire function for which infinitely many 
   Fatou components share the same boundary. This resolves the long-standing
    problem whether \emph{Lakes of Wada continua}
     can arise in complex dynamics, and
     answers the analogue of a question of Fatou from 1920 concerning Fatou components
     of rational functions. 
   \item We answer a question of Rippon concerning the existence of non-escaping points
     on the boundary of a bounded escaping wandering domain, that is, a wandering Fatou 
     component contained in the escaping set. In fact, we show that the set of
    such points can have positive Lebesgue measure. \label{item:ripponquestion}
       \item We give 
       the  first example of an entire function having a 
      simply connected Fatou component 
       whose closure has a disconnected complement, answering a question of 
       Boc Thaler.
       \end{enumerate}
In view of~\ref{item:ripponquestion}, we introduce the concept of \textit{maverick points}: points  on the boundary of  a wandering domain whose 
  accumulation 
  behaviour differs from that of internal points. We prove that the set of such points has 
  harmonic measure zero, but that it can nonetheless be rather large. For example, it may have positive planar Lebesgue measure. 
\end{abstract}

\maketitle

\section{Introduction}

The iteration of transcendental entire self-maps of the complex plane was initiated by Fatou in 1926~\cite{fatou26} and has received much attention in recent years. 
 A central object, introduced by Eremenko~\cite{eremenko89} in 1989, is the \emph{escaping set} of a transcendental entire function $f$, 
     \[\I(f) \defeq \bigl\{z\in\C\colon f^n(z)\to\infty\text{ as }n\to\infty\bigr\}.\]
    The escaping set is significant in the study of transcendental dynamics since it often contains structures that can be used to understand the global dynamics of $f$.
    Indeed, already Fatou \cite[p.~369]{fatou26} 
    noticed the existence of curves to infinity in the escaping sets of certain
    transcendental entire functions. Eremenko proved that,
    for every transcendental entire function, the connected components of $\overline{I(f)}$ are all unbounded \cite[Theorem~3]{eremenko89}. 
He also states~\cite[p.~343]{eremenko89}
 that it is plausible that  the set $I(f)$ itself has no bounded connected components; 
  this is known as \emph{Eremenko's conjecture} and has remained an open problem since then; see also~\cite[Question 15]{bergweiler93}.
\begin{conjecture}[Eremenko's conjecture]
\label{conj:eremenko}
Every connected component of the escaping set of a transcendental entire function 
is unbounded. 
\end{conjecture}     

 The escaping set and Eremenko's conjecture have been at the centre of much of the progress that transcendental dynamics has seen during the past two decades. 
  For example, 
  Rippon and Stallard~\cite{rippon-stallard05}
   showed that a certain subset of $I(f)$ that had been introduced by Bergweiler and Hinkkanen \cite{bergweiler-hinkkanen99}, the \emph{fast escaping set} $A(f)$, 
   see~\eqref{eqn:fastescapingset}, has only unbounded components. (In particular, there is always at least one unbounded connected component of $I(f)$.)
     The fast escaping set has since become an important object in the field. 
     In particular, it plays an important role
    in Bishop's construction of a transcendental entire function whose Julia set has Hausdorff 
    dimension one~\cite{bishop18}; this solved a long-standing open question. 
    Work on the fast escaping set has also had a substantial impact in quasiregular 
    dynamics; see~\cite{quasiregularescaping,quasiregularfastescaping}. 
  
 Furthermore, investigation of the escaping set has led to major advances in the study of two
   important classes of transcendental entire functions: the 
   \emph{Eremenko--Lyubich class} $\B$ 
    and the \emph{Speiser class} $\mathcal{S}$. These 
  consist of those transcendental entire functions $f$ for which the set $S(f)$ of \emph{singular values} is   
   bounded resp.\ finite; compare~\cite{eremenko-lyubich87}. 
  For instance, in~\cite{rrrs11}, it was shown that a stronger version of Eremenko's conjecture, also stated in~\cite{eremenko89}~--
    that every escaping point can be connected to infinity by a curve consisting of escaping points~-- holds for a large subclass of $\B$, but
   fails for general $f\in\B$. 
   According to Bishop (personal communication), the question whether such 
    counterexamples can also be constructed in~$\mathcal{S}$ was one of the original
    motivations for his technique of quasiconformal folding~\cite{bishop15}, 
    which has revolutionised the study of the classes $\B$ and $\mathcal{S}$.
   
   Ideas introduced to study the escaping set have had significant impact also outside of transcendental dynamics; for example,
     techniques from~\cite{rrrs11} and~\cite{rempe09} were applied by 
     Dudko and Lyubich~\cite{dudkolyubichmandelbrot} to achieve progress on the question of local connectivity of the Mandelbrot set, while quasiconformal folding
     has led to new results on the geometry of Riemann surfaces~\cite{bishoprempetriangles}. 
          
Eremenko's conjecture has been
 confirmed in a number of cases, while stronger versions have been disproved. In addition to the results
 already referred to above, we mention~\cite{schleicherzimmer03,rempe07b,ripponstallard13,osborneetal19,nicksetal21} and~\cite[Theorem~1.6]{arclike}. 
  On the other hand, it is known that for an entire function $f$, $\I(f)\cup \{\infty\}$ is
 always connected~\cite{rippon-stallard11}, and that unbounded connected components are dense in $\I(f)$~\cite{rippon-stallard05}. 
 Moreover, there is a \emph{quasiregular} map $f\colon\C\to\C$ such that $\overline{\I(f)}$ has bounded components \cite{quasiregularescaping},
   in contrast to the entire case. So the problem is not of a purely topological nature.

The above results show that
  any
 counterexample to Eremenko's conjecture has to be quite subtle, and explain why it has remained open  until now
 despite considerable research
 efforts. In this article, we overcome these difficulties and resolve the conjecture in the negative, constructing a wide range of counterexamples. Recall that a compact set $X\subseteq \C$ is \emph{full} if $\C\setminus X$ is connected.
 
 \begin{thm}[Counterexamples to Eremenko's conjecture]\label{thm:eremenkocounterexample}
   Let $X\subseteq\C$ be a non-empty full and connected compact set. Then there exists a transcendental entire function $f$ such that $X$ is a connected
     component of $\I(f)$. 
 \end{thm}

We remark that a bounded connected component of $I(f)$ need not be compact, but any compact connected component is necessarily full. See Proposition~\ref{prop:compact-comp}.

The examples in Theorem~\ref{thm:eremenkocounterexample} are constructed using approximation theory; more precisely, we use
 a famous theorem of 
 Arakelyan on approximation by entire functions 
 (see Theorem~\ref{thm:arakelyan}). The use
 of approximation theory to construct examples in transcendental dynamics has 
 a long history, going back to work of Eremenko and 
 Lyubich from 1987~\cite{eremenko-lyubich87}. We shall prove Theorem~\ref{thm:eremenkocounterexample}
   by developing a new general method for prescribing the dynamical behaviour of an entire function on certain closed
  subsets of the complex plane. We use this procedure to answer a number of further open problems in transcendental dynamics, which we describe in the following. 
  
  Let $f\colon S\to S$ be a transcendental entire
  function or a rational map, where
    $S=\C$ or $S=\Ch\defeq \C\cup\{\infty\}$,
    respectively. The \emph{Fatou set} $\F(f)\subseteq S$ is the set of points whose orbits under $f$
   remain stable under small perturbations.
 More formally, $\F(f)$ is the largest open set on which the iterates of $f$ are equicontinuous with respect
 to spherical distance; equivalently, it is the largest forward-invariant open set that omits at least three points of the Riemann sphere
  \cite[Section~2.1]{bergweiler93}. 
  The
  complement \mbox{$\J(f) \defeq S\setminus \F(f)$}  is called the \emph{Julia set}.
  
These two sets are named after Pierre Fatou and Gaston Julia, who independently laid the foundations of
 one-dimensional complex dynamics in the early 20th century. In his seminal memoir on the iteration
 of rational functions, 
 Fatou posed the following question concerning the structure of the connected components of $\F(f)$, which are 
 today called the \emph{Fatou components} of $f$.

\begin{question}[{Fatou \cite[pp.~51--52]{fatou20}}]
\label{qu:fatou}
If $f$ has more than two Fatou components, can two of these components share the same boundary?\footnote{%
  ``S'il y a plus de deux r\'egions contigu\"es distinctes (et par suite une infinit\'e), deux r\'egions contigu\"es peuvent-elles
   avoir la m\^eme fronti\`ere, sans \^etre identiques?''}
\end{question}
 
Fatou asked this question in the context of rational self-maps of the Riemann sphere, but
it makes equal sense
 for transcendental entire functions, whose dynamical study Fatou initiated in 1926 \cite{fatou26}. We give
 a positive answer to Question~\ref{qu:fatou} in this~setting.
 
\begin{figure}
\includegraphics[width=.54\linewidth]{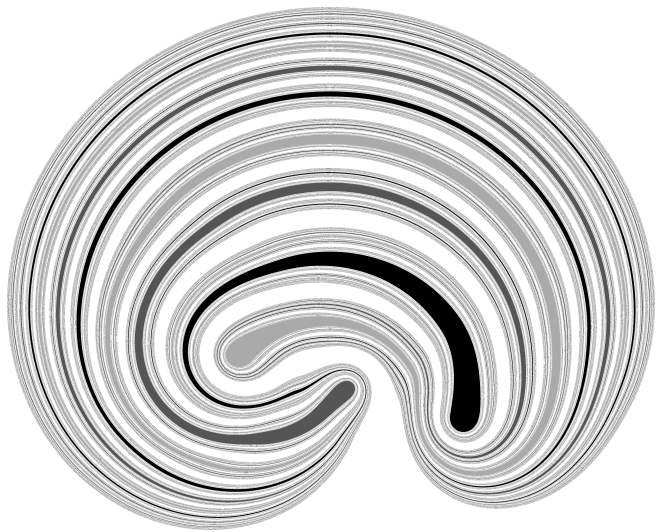} 
\caption{A Lakes of Wada continuum bounding four domains.}
\label{fig:lakes-of-wada}
\end{figure}

\begin{thm}[Fatou components with a common boundary]
\label{thm:fatou}
There exists a transcendental entire function $f$ and an infinite collection of Fatou components of~$f$ that all share the same boundary. In particular, the common boundary of these Fatou components is a Lakes of Wada continuum.
\end{thm}

Here a \emph{Lakes of Wada
 continuum} is 
   a compact and connected subset of the plane that is the common boundary of three 
   or more disjoint domains (see Figure~\ref{fig:lakes-of-wada}).
  That such continua exist was first shown by Brouwer~\cite[p.~427]{brouwer10};
   the name ``Lakes of Wada'' arises from a well-known construction that 
   Yoneyama~\cite[p.~60]{yoneyama17} attributed to his advisor, 
   Takeo Wada; compare \cite[p.~143]{hocking-young61},~\cite[Section~8]{hubbard-oberstevorth95} or~\cite[Section~4]{Ishii2017}. 
   In this construction, one begins with the closure of a bounded finitely connected domain, which we may think of as an island in the sea. 
   We think of the complementary domains, which we assume to be
   Jordan domains, 
   as bodies of water: the sea (the unbounded domain) and the lakes (the bounded domains).
    One then constructs 
    successive canals, which ensure that the maximal distance from any point on the remaining piece of land to any body of water
    tends to zero, while keeping the island connected; see Figure~\ref{fig:island}. The land remaining at the end of this 
    process is the boundary of each of its complementary domains; in particular, if there are at least three bodies
    of water, it is
      a Lakes of Wada continuum. One may even obtain infinitely many 
    complementary domains, all sharing the same boundary, by introducing new lakes throughout
    the construction.

Lakes of Wada continua may appear pathological, but they occur naturally in the study of 
dynamical systems in two real dimensions; see \cite{kennedy-yorke91}. For example, Figure~\ref{fig:lakes-of-wada} is (a projection of) the Plykin attractor \cite{plykin74} (see also \cite{coudene06}), which is the attractor of a perturbation of an Anosov diffeomorphism of the torus. Hubbard and Oberste-Vorth \cite[Theorem~8.5]{hubbard-oberstevorth95} showed that, under certain circumstances, the basins of attraction of 
H\'enon maps in $\mathbb{R}^2$ form Lakes of Wada.

Theorem~\ref{thm:fatou} provides the first example where such continua
 arise in one-dimen\-sio\-nal \emph{complex} dynamics, 
 answering a long-standing open question.
 The second author
 learned of this problem for the first time in an introductory complex dynamics course
 by Bergweiler in Kiel during the academic year 1998--99; see~\cite[p.~27]{bergweiler-notes}.
  
  \begin{figure}
\includegraphics[width=\linewidth]{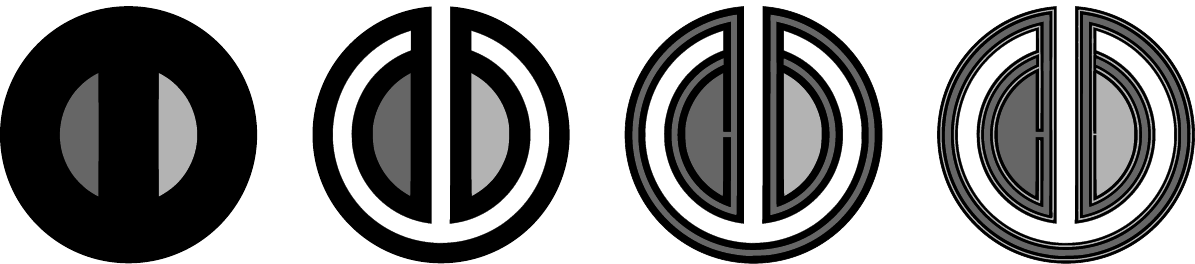} 
\caption{First steps in the construction of a Lakes of Wada continuum. The island is black, the sea is white,
  and the two lakes are dark grey and light grey. The construction proceeds
   inductively, with the distance from any point in the island to any 
   body of water tending to zero.}
\label{fig:island}
\end{figure}

 A Fatou component $U$ is called \textit{periodic} if $f^p(U)\subseteq U$ for some $p\in\N$, and $U$~is \textit{preperiodic} if there exists $q\in\N$ such that $f^{q}(U)$ is contained in a periodic Fatou component; otherwise $U$ is called a \textit{wandering domain} of $f$. The work of Fatou and Julia left open the question
whether wandering domains exist. In 1976, Baker \cite{baker76} showed that transcendental entire functions can have wandering domains, while
 Sullivan~\cite{sullivan85} proved in 1985 that they do not occur for rational functions.
If $f$ is a polynomial, a bounded periodic Fatou component $U$ is either 
 an immediate basin of a finite attracting or parabolic periodic point, or a \emph{Siegel disc}, on which the dynamics is conjugate to rotation by an irrational
 angle. 
 In the former case, it is known that $U$ is a Jordan domain \cite{roesch-yin08}.
 Less is known about the possible topology of Siegel disc boundaries. 
   However, suppose that $f$ is a 
  \emph{quadratic}
  polynomial with a Siegel disc $\Delta$ of rotation number $\alpha$. Petersen and Zakeri~\cite{petersenzakeri04} proved that $\partial \Delta$ is a Jordan curve
   when $\alpha$ belongs to a positive measure set of rotation numbers. 
    Shishikura and Yang~\cite{shishikurayang} and Cheraghi~\cite{cheraghitopology} 
     showed that the same holds for all \emph{high-type} rotation numbers
     (for which all coefficients of the continued fraction expansion exceed a certain
     universal constant).
     Dudko and Lyubich~\cite[Theorem 1.2]{dudkolyubichsiegel} recently announced that, 
     without restriction on the rotation number $\alpha$, the function $f$ is injective on the 
    orbit of $\overline{\Delta}$. This rules out 
    Lakes of Wada boundaries, and provides a negative answer to Question~\ref{qu:fatou}    
     for quadratic polynomials.  Both problems remain
    open for polynomials of degree at least three, but it seems reasonable to expect that, for 
    polynomials and rational maps, the 
 answer to Question~\ref{qu:fatou} is negative. The 
 problem whether the \emph{whole} Julia set can be a Lakes of Wada continuum is 
 related to \emph{Makienko's conjecture} concerning completely invariant domains
 and buried points of rational functions; compare~\cite{sun-yang03} and~\cite{cmmr09}. 
 
In contrast, our motivation for considering 
   Question~\ref{qu:fatou} comes from the study of wandering domains,
    which have been the focus of much recent 
   research; compare~\cite{bishop15,benini-rippon-stallard16,bishop18,martipete-shishikura20,benini-fagella-evdoridou-rippon-stallard}. They    have also recently been
   studied in higher-dimen\-sional complex dynamics; compare \cite{abdpr16,bocthaler21b,hahn-peters21,arosio-benini-fornaess-peters19}. A wandering domain $U$ of an entire function $f$
   is called \emph{escaping} if 
   it is contained in the escaping set $\I(f)$ 
    and it is called \emph{oscillating} if it is contained in the \emph{bungee set}
    \[\BU(f) \defeq \bigl\{z\in\C\setminus \I(f)\colon \limsup_{n\to\infty} |f^n(z)| = +\infty \bigr\}. \] 
   
     We are interested in establishing
        whether the behaviour of points on $\partial U$ is determined
        by that of points in $U$. 
        In all previously known examples of escaping wandering domains, the iterates
        $f^n$ in fact tend to infinity \emph{uniformly} on $U$, and hence on $\partial U$. 
        It is known \cite[Theorem~1.2]{rippon-stallard12} that this must be the case whenever
        $U$ is contained in the fast escaping set $A(f)$. 
       This suggests the following question, which is Problem 2.94 in Hayman and Lingham's 
        \emph{Research Problems in Function Theory}.
        
   \begin{question}[Rippon {\cite[p.~61]{hayman-lingham19}}]\label{qu:rippon}
       If $U$ is a bounded escaping wandering domain of a transcendental entire function $f$, is $\partial U\subseteq \I(f)$? 
   \end{question}

  Here and subsequently, a \emph{bounded wandering domain} is a 
    wandering domain $U$ that is bounded as a subset of the complex plane. 
    One should note the distinction with \emph{orbitally bounded} wandering domains:
    wandering domains whose points have bounded orbits. The existence of orbitally
    bounded wandering domains is a famous open problem; see~\cite[Problem~2.67]{hayman-lingham19}.
     We answer Question~\ref{qu:rippon} in the negative. 
   \begin{thm}[Non-escaping points in the boundary of escaping wandering domains]
   \label{thm:rippon}
      There exists a transcendental entire function $f$ with a bounded escaping
      wandering domain $U$ such that $\partial U\setminus \I(f)\neq\emptyset$.
   \end{thm}
            We prove Theorem~\ref{thm:rippon}    
    by showing that the function
     from Theorem~\ref{thm:fatou} can be constructed such that 
      at least one of these domains, $U_1$,
      is an escaping wandering domain while another, $U_2$, is an 
      oscillating wandering domain. 
     The set of non-escaping points in $\partial U_1 = \partial U_2$ 
      has full harmonic measure seen from $U_2$ by a result of 
      Rippon and Stallard~\cite[Theorem~1.2]{rippon-stallard11}; in particular, it is non-empty.
      
   Both Theorem~\ref{thm:fatou} and Theorem~\ref{thm:rippon} are
     consequences of a much more general result, which concerns wandering
     compact sets of transcendental entire functions having arbitrary shapes. 
\begin{thm}[Entire functions with wandering compacta]
\label{thm:main}
Let $K\subseteq \C$ be a full compact set. Let
  $Z_{I},Z_{BU}\subseteq K$ be disjoint finite or countably infinite sets such that 
  no connected component of $\interior(K)$ intersects both 
  $Z_{\I}$ and $Z_{\BU}$. Then there exists a transcendental entire function $f$ such that 
   \begin{enumerate}[(i)]
   \item
   $\partial K\subseteq \J(f)$;\label{item:boundaryinJ}
     \item $f^n(K)\cap f^m(K)=\emptyset$ for $n\neq m$;\label{item:wanderingcompactum}
      \item every connected component of $\interior(K)$ is a wandering domain of $f$;\label{item:wanderingdomains}
      \item
   $Z_{\I}\subseteq \I(f)$ and $Z_{\BU}\subseteq \BU(f)$.\label{item:maverick}
   \end{enumerate}
\end{thm}
     
 Theorem~\ref{thm:main} is inspired by a previous
  result of Boc Thaler~\cite{bocthaler21}. Let $U\subseteq \C$ be a bounded domain such that $\C\setminus \overline{U}$ is connected,
    and such that furthermore $U$ is \emph{regular} in the sense that $U = \interior(\overline{U})$. Using a version of Runge's 
    approximation theorem, Boc Thaler proved that there is a transcendental entire function~$f$ for which $U$ is a wandering domain on which the iterates of $f$ tend to~$\infty$ \cite[Theorem~1]{bocthaler21}. He then asked the following. 

\begin{question}[{Boc Thaler \cite[p.~3]{bocthaler21}}]
\label{qu:boc-thaler}
Is it true that the closure of any bounded simply connected Fatou component of an entire function has a connected complement? 
\end{question}

 If $U$ is a bounded domain such that $\partial U$ is a Lakes of Wada continuum, then the complement of $\overline{U}$ has at least two
   connected components by definition. Hence 
   the closure of the Fatou component from Theorem~\ref{thm:fatou} has a disconnected complement, answering Question~\ref{qu:boc-thaler} in the negative.

 \begin{rmk}\label{rmk:Jcomponents}
\begin{enumerate*}[(i)]
\item\label{item:Jcomponents} We do not know whether in Theorem~\ref{thm:main} one can additionally ensure that every connected component of $\partial K$ is a connected component of $J(f)$. This would only be possible if $f$ had multiply connected wandering domains limiting on $\partial K$.  \\

\item Certain \textit{unbounded} domains can also be realised as wandering domains using similar techniques; see Remark~\ref{rmk:unbounded-wd}     for the case of a half-strip.
\end{enumerate*}
 \end{rmk}

   In the case where $Z_{\BU}=\emptyset$ (which suffices 
   for Theorem~\ref{thm:fatou}), 
   we can prove Theorem~\ref{thm:main} by a similar proof as 
   Boc Thaler's, but applying a subtle change of point of view: 
   instead of beginning with a simply connected domain and approximating its closure, as in~\cite{bocthaler21}, 
   we start the construction with the full compact set $K\subseteq \C$ and consider
   its interior components. (See Theorem~\ref{thm:uniform}.) In this case, we can even ensure that $K$ belongs to the fast escaping set $A(f)$;
   see Proposition~\ref{prop:fastescaping}. The only previously known examples of fast escaping bounded 
   simply connected wandering domains are due to Bergweiler~\cite{bergweiler11} and Sixsmith~\cite{sixsmith12}. 
   This construction can be used to 
   provide new counterexamples to the strong Eremenko conjecture, mentioned above. It was resolved in the negative
   in~\cite[Theorem~1.1]{rrrs11}, but our construction is much simpler. (See Theorem~\ref{thm:strongeremenko}.)
   
  The iterates of the function $f$ resulting from this argument converge
    to infinity uniformly on~$K$, so additional ideas are needed to prove Theorems~\ref{thm:rippon}
    and~\ref{thm:main}. A key new ingredient in the proof is to use Arakelyan's theorem
    (instead of Runge's theorem) to ensure that there is a sequence of unbounded domains
     that are mapped conformally over one another by any
    function involved in the construction (see Section~\ref{sec:scaffolding}
    and Figure~\ref{fig:strips}). The presence of these domains allows us to 
    let the image of the compact set $K$ be stretched (horizontally) at certain steps 
    during the construction, ensuring that
    the spherical diameter of $f^n(K)$ does not tend to zero as $n\to\infty$, and allowing $K$
    to contain points of both the escaping set and the bungee set. 
    
  In order to obtain a counterexample to Eremenko's conjecture, 
   we develop the method even further, now applying it to an \emph{unbounded} ray 
    $K$ 
    that connects a finite endpoint to~$\infty$. The unboundedness of $K$ allows
    us to ensure that $K$ is surrounded by connected components
    of an attracting basin, which means that any connected component
    of the escaping set intersecting $K$ is also contained in $K$. By additionally
    ensuring that the finite endpoint of $K$ escapes, while some other point
    is in $\BU(f)$, we obtain a counterexample to Eremenko's conjecture. 
    The stronger statement in Theorem~\ref{thm:eremenkocounterexample} is obtained 
    by a more careful application of similar ideas. 

Recall that the example from Theorem~\ref{thm:rippon} has an escaping wandering domain 
  and an oscillating wandering domain, and their (shared) boundary contains both escaping points and bungee
  points. It seems interesting to investigate, more generally, those points on the boundary of a wandering domain
  whose orbits have different accumulation behaviour than the interior points. 

\begin{dfn}[Maverick points]
\label{dfn:maverick-points}
Let $f$ be a transcendental entire function and suppose that $U$ is a wandering domain of $f$. We say that a point $z\in \partial U$ is \textit{maverick} if there is a sequence $(n_k)$ such that $f^{n_k}(z)\to w\in\Ch$ as $k\to\infty$, but $w$ is not a limit function of $f^{n_k}|_U$. 
\end{dfn}

Equivalently, a point $z\in \partial U$ is maverick if for any (and hence all) $w\in U$, one has\linebreak
   $\limsup_{n\to\infty} \dist^{\#}(f^n(z),f^n(w))>0$, where $\dist^{\#}$ denotes spherical distance. (Compare Lemma~\ref{lem:maverick}.) 
   If $U$ is an escaping wandering domain, then $z\in \partial U$ is maverick if and only if $z\notin \I(f)$. However,
   if $U$ is oscillating, then $\partial U$ may contain maverick points that are also in the bungee set
   $\BU(f)$, but with a different accumulation pattern. 
    Indeed, the function $f$ that is constructed in the
     proof of Theorem~\ref{thm:main} satisfies 
     \[ \dist^{\#}(f^n(z),f^n(w))>0 \]
    for any distinct $z,w\in Z_{\BU}$ not belonging to the same interior component of $K$; see Remark~\ref{rmk:BUmaverick}. 
    In particular, if $z$ was chosen in some component $U$ of $\interior(K)$ and 
    $w\in \partial U$, then $w$ will be a maverick point of the wandering domain $U$.

  As far as we are aware, 
   Theorems~\ref{thm:rippon} and~\ref{thm:main} provide the first examples of wandering domains whose
   boundaries contain maverick points. 
   The following result shows that, as suggested by their name, most points on $\partial U$ are not maverick points. 

\begin{thm}[Maverick points have harmonic measure zero]
\label{thm:maverick}
Let $f$ be a transcendental entire function and suppose that $U$ is a wandering domain of $f$. The set of maverick points in $\partial U$ has harmonic measure zero with respect to $U$.
\end{thm}

 If $U\subseteq \I(f)$, Theorem~\ref{thm:maverick} reduces to
   \cite[Theorem~1.1]{rippon-stallard11}. 
For non-escaping wandering domains, the theorem strengthens 
\cite[Theorem~1.3]{osborne-sixsmith16}, which states that, for points $z\in \partial U$ from
  a set of full harmonic measure, the $\omega$-limit set $\omega(z, f)=\bigcap_{n=1}^\infty \overline{\{f^k(z)\colon k>n\}}$ agrees with the $\omega$-limit set 
  of points in $U$. Observe that maverick points may have the same $\omega$-limit set as points in $U$. Indeed, the function in Theorem~\ref{thm:main} can be constructed so that all points in $Z_{\BU}$ share
  the same $\omega$-limit set, but have different accumulation behaviour; see Remark~\ref{rmk:omegalimit}. 
    So the set of maverick points may indeed be larger than the set considered
   by Osborne and Sixsmith. 

 Benini et al~\cite[Theorem~9.3]{benini-fagella-evdoridou-rippon-stallard2} independently
   prove a result that implies the following weaker version of Theorem~\ref{thm:maverick}:
     if $U$ is a wandering domain of a transcendental entire function, and
     $Z\subseteq\partial U$ is a set of maverick points such that
     additionally $\dist^{\#}(f^n(z),f^n(w))\to 0$ as $n\to\infty$ for all $z,w\in Z$, then $Z$ has
     harmonic measure zero. However, their result applies in a setting (sequences of holomorphic
     maps $F_n\colon U\to U_n$ from a simply connected domain $U$ to simply connected domains $U_n$) in which the 
     stronger conclusion of Theorem~\ref{thm:maverick} becomes false in general;
     compare \cite[Example~7.6]{benini-fagella-evdoridou-rippon-stallard2}. 
     
  Generalising a question of Bishop
    \cite[p.~133]{bishop14}
    for escaping wandering domains, we may ask whether 
    Theorem~\ref{thm:maverick} can be strengthened as follows. 
(See~\cite[Section~5.2]{bishop14} for the definition and a discussion of logarithmic capacity.) 
        
       \begin{question}
          Let $U$ be a simply connected 
            wandering domain of a transcendental entire function. Does the set of maverick points in
            $\partial U$ have zero logarithmic capacity when seen from $U$? That is, let $\phi\colon\DD\to U$ be a conformal isomorphism between 
             the unit disc $\DD$ and $U$, and consider the set $\Xi\subseteq \partial\DD$
             of points at which the radial limit of $\phi$ exists and is a maverick point.
             Does $\Xi$ have zero logarithmic capacity? 
       \end{question}

 Our proof of Theorem~\ref{thm:main} can be adapted to show that logarithmic capacity
   zero is the best one may hope for in the above question.
   
  \begin{thm}[Maverick points of logarithmic capacity zero]\label{thm:capacity}
   Let $\Xi\subseteq\partial\DD$ be a compact set of zero logarithmic capacity. Then
     there exists a  transcendental entire function $f$ such that $\DD$ is a wandering domain of $f$, and
       every point in $\Xi$ is maverick. The wandering domain in this example
        can be chosen to be either escaping or oscillating. 
  \end{thm}       
 
   Sets of zero harmonic measure (or zero logarithmic capacity) need not be small in an absolute geometric sense. Indeed, recall that the set of maverick points in our
    proof of Theorem~\ref{thm:rippon} has full harmonic measure when seen from
    another complementary domain of the Lakes of Wada continuum, and hence
    has Hausdorff dimension at least $1$ by a result of Makarov;
     see~\cite[Section~VIII.2]{garnett-marshall05}. By a further 
    application of our construction, we can strengthen the example as follows.
    
    \begin{thm}[Maverick points of positive Lebesgue measure]\label{thm:HD}
      There exists a transcendental entire function $f$ with a         wandering domain $U$ such that the set of non-maverick points on $\partial U$
        has Hausdorff dimension $1$, while the set of maverick points contains
        a continuum of positive Lebesgue measure. The wandering domain in this example
        can be chosen to be either escaping or oscillating. 
    \end{thm}
     
\subsection*{Further questions} 
     We conclude the introduction with several questions arising from our work.   
       While Eremenko's conjecture is false in general, 
       it is shown in~\cite[Theorem~1.6]{rrrs11} that it holds for all functions of finite
       order in the \emph{Eremenko--Lyubich class}~$\B$. (See~\cite{rrrs11} for definitions.)
       We may ask whether one of these two conditions can be omitted. 
     
     \begin{question}
       Can $I(f)$ have a bounded connected component if $f\in \mathcal B$ has infinite
          order,
          or if $f\notin\B$ has finite order? 
     \end{question}

The function $f$ constructed in the proof of Theorem~\ref{thm:eremenkocounterexample} is of infinite order and does not belong to the class $\mathcal B$ (see Remark~\ref{rem:order-classB}).
     
The following question is a modification of Question~\ref{qu:boc-thaler} that takes into account the new types of wandering domains provided by Theorem~\ref{thm:main}. Recall that if $A\subseteq\C$
   is compact, then $\Fill(A)$ denotes the \emph{fill} of $A$, that is, the complement of the
   unbounded connected component of $\C\setminus A$.  

\begin{question}\label{qu:bocthalernew}
Suppose that $U$ is a bounded simply connected Fatou component of a transcendental entire function, and let $K=\Fill(\overline{U})$. Is it true that $\partial U = \partial K$?
\end{question}  

In view of Theorem~\ref{thm:main}, a positive answer to Question~\ref{qu:bocthalernew} would imply that a bounded
simply connected domain $U$ can arise as a Fatou component of an entire function if 
and only if
$\partial U = \partial \Fill(\overline{U})$.

Our results imply the existence of both escaping and oscillating simply connected
   wandering
   domains with Lakes of Wada boundaries. In contrast, it appears 
   much more difficult to construct 
   \emph{invariant} Fatou components with Lakes of Wada boundaries, if this is at all possible. 
   
\begin{question}
\label{qu:invariant}
Let $f$ be a transcendental entire function and suppose that $U$ is an invariant Fatou component of $f$. Must every connected component of
  $\C\setminus \overline{U}$ intersect~$\J(f)$? 
\end{question}

If a Fatou component $U$ is completely invariant (i.e.\ $f^{-1}(U)=U$), then 
$\partial U = J(f)$. Hence 
a positive answer to Question~\ref{qu:invariant} would imply
that a transcendental entire function has at most one such Fatou component;
 compare~\cite{rempegillen-sixsmith19}.
The previously-mentioned partial results in the polynomial case motivate the following strengthening of Question~\ref{qu:invariant} for \emph{bounded} $U$.
\begin{question}
Let $f$ be an entire function and suppose that $U$ is a bounded invariant Fatou component of $f$. Must $\partial U$ be a simple closed curve?
\end{question}

\subsection*{Meromorphic functions} For 
  transcendental \emph{meromorphic} functions $f\colon \C\to\Ch$, 
   the analogue of Eremenko's conjecture is false for elementary reasons.
   Indeed, for the function $f(z)=\tan(z)/2$, we have $I(f)\subseteq J(f)$ and $J(f)$ is totally 
   disconnected; see e.g.~\cite[Theorem~5.2]{keenkotus97}. However, the study of wandering domains
   of transcendental meromorphic functions
   is of significant 
   interest. All entire functions constructed in this paper can easily be modified
   to be meromorphic with infinitely many poles. 
   Even more is true: any (not necessarily full) compact set $K$ can be 
   realised as a wandering compactum of a transcendental meromorphic function
   $f$, in such a way that every connected component of $\partial K$ is 
   a connected component of $J(f)$; see~\cite[Theorem~1.3]{mrw2} and recall Remark~\ref{rmk:Jcomponents}~\ref{item:Jcomponents}. 
   For a meromorphic function having a wandering domain whose orbit consists only of simply connected Fatou components, 
   Theorem~\ref{thm:maverick} goes through with the same
   proof. It is plausible
   that Theorem~\ref{thm:maverick} is true also for general 
   wandering domains of transcendental meromorphic functions, but this
   requires further investigation.
   
\subsection*{Notation} 
For a set $X\subseteq \C$, let $\partial X$, $\operatorname{int}(X)$, and $\overline{X}$ denote, respectively, the boundary, the interior, and the closure of $X$ in~$\C$. We write $\dist$, $\dist^\#$ and $\dist_U$ 
for the Euclidean, spherical and hyperbolic distance in a domain $U\subseteq\C$, respectively. We use $\ell_U(\gamma)$ to denote the hyperbolic length of a curve $\gamma$ in a domain $U\subseteq \C$. Moreover, $\diam$ and $\diam^\#$ denote the Euclidean and spherical diameter, respectively. The Euclidean disc of radius $\delta>0$ around $z\in\C$ is denoted by $\D(z,\delta)$, and
the unit disc is denoted by $\DD \defeq \D(0,1)$. Finally, we use $\mathbb{H}\defeq\{z\in\C\colon \im z>0\}$.

\subsection*{Structure of the paper} In Section~\ref{sec:preliminaries}, we recall the results from approximation theory needed for our constructions, and prove a 
number of useful lemmas about behaviour that is preserved under approximation. In Section~\ref{sec:uniform}, 
  we give a simplified proof of Theorem~\ref{thm:main} in the special case that the wandering compact set $K$ escapes uniformly, that is, $Z_{BU}=\emptyset$ (see Theorem~\ref{thm:uniform}). We use this to obtain wandering domains that form Lakes of Wada (Theorem~\ref{thm:fatou}) and new counterexamples to the strong
  Eremenko conjecture (Theorem~\ref{thm:strongeremenko}). The results of Section~\ref{sec:uniform} are not strictly required for the proofs of the main results stated in the introduction 
  (Theorem~\ref{thm:fatou} also follows from Theorem~\ref{thm:main}), but the proof of Theorem~\ref{thm:uniform} already contains a number of ideas that are also
  present in the more involved constructions relating to Theorems~\ref{thm:main} and~\ref{thm:eremenkocounterexample}.

In Section~\ref{sec:scaffolding}, we set up the basic structure of the examples constructed in both Theorem~\ref{thm:main} and Theorem~\ref{thm:eremenkocounterexample}. 
  The proof of Theorem~\ref{thm:main} in full generality, and the deduction of Theorem~\ref{thm:rippon} from it, is in Section~\ref{sec:maverickconstruction}. 
  Section~\ref{sec:mavericklarge} briefly outlines how the proof of Theorem~\ref{thm:main} can be modified to prove Theorems~\ref{thm:capacity} and~\ref{thm:HD}. 
  The techniques and results of Section~\ref{sec:mavericklarge} are not used elsewhere in the paper.

   Our counterexamples to Eremenko's conjecture are constructed in Section~\ref{sec:eremenko}. This construction is similar to 
   that in 
   Section~\ref{sec:maverickconstruction}, but can be read independently of it. Thus, a reader primarily interested in the
   proof of Theorem~\ref{thm:eremenkocounterexample} should study Sections~\ref{sec:preliminaries},~\ref{sec:scaffolding} and~\ref{sec:eremenko}, but could choose to
   omit Sections~\ref{sec:uniform},~\ref{sec:maverickconstruction} and~\ref{sec:mavericklarge}.
   
   Theorem~\ref{thm:maverick} is proved in Section~\ref{sec:maverick}. We conclude the paper with Section~\ref{sec:spikes}, which discusses and proves an auxiliary result
     used in Sections~\ref{sec:maverickconstruction} and~\ref{sec:mavericklarge}.

\subsection*{Acknowledgements} We thank Walter~Bergweiler, Luka~Boc Thaler, 
Jack~Burkart, Alex~Eremenko, Vasiliki~Evdoridou, Misha~Lyubich, Phil~Rippon, Mitsuhiro~Shishikura, Dave~Sixsmith and Gwyneth~Stallard for interesting discussions on this topic. We are particularly grateful to Alex~Eremenko for mentioning to us
the question of Fatou in~\cite{fatou20}. Finally, we thank the referee for their careful reading of the manuscript and many thoughtful comments that improved the quality of the exposition.

\subsection*{Note from the authors} 
This work follows in the footsteps of two world-leading Ukrainian mathematicians, Alex Eremenko and Misha Lyubich. Their collaboration, which pioneered the use of approximation theory in complex dynamics, took place in the fall of 1983 in Kharkiv. At that time, Alex Eremenko was based at the Institute of Low Temperature Physics and Engineering, and it was there that he formulated what is now known as Eremenko's conjecture. The city of Kharkiv has been devastated during the ongoing invasion of
Ukraine, and the Institute  has been severely damaged \cite{kharkiv}. We dedicate this paper to Profs.\ Eremenko and Lyubich,
  to the people of Kharkiv, and to all victims of the invasion of Ukraine.

\section{Preliminary results on approximation}

\label{sec:preliminaries}

We recall two classical results of approximation theory. The classical theorem 
 of Runge~\cite{runge85} (see also \cite[Theorem~2 in Chapter II~\S 3]{gaier87}) 
 concerns the approximation of functions on compact and full sets by polynomials
 \footnote{%
  Runge's theorem in fact applies more generally to approximation of
    holomorphic functions on arbitrary compact sets by
    \emph{rational} maps~\cite[Theorem~2 in Chapter III~\S 1]{gaier87}. The key case for us is when the compact set
    is full, in which case the approximating function can be chosen to be a polynomial.}.
   This is sufficient to prove Theorem~\ref{thm:main} when $Z_{\BU}=\emptyset$,
   and hence Theorem~\ref{thm:fatou}. 
   In contrast, 
  Arakelyan's theorem \cite{arakelyan64} (see also \cite[Theorem~3 in Chapter~IV~\S 2]{gaier87}) allows us to approximate
  functions defined on (potentially) unbounded sets by entire functions; this will
  be crucial in our remaining constructions and, in particular, in the proof of 
  Theorem~\ref{thm:eremenkocounterexample}.  

\begin{thm}[Runge's theorem on polynomial approximation]
 \label{thm:runge}
  Let $A\subseteq \C$ be a compact set such that $\C\setminus A$ is connected. 
 Suppose that $g\colon A\to \C$ is a continuous function that extends holomorphically to an open neighbourhood of $A$. 
  Then for every $\eps>0$, there exists a polynomial~$f$ such that 
\[
\lvert f(z) - g(z)\rvert < \eps\quad \text{for all } z\in A.
\] 
\end{thm}

\begin{thm}[Arakelyan's theorem]
\label{thm:arakelyan}
Let $A\subseteq \C$ be a closed set such that
\begin{enumerate}[(i)]
\item $\CR\setminus A$ is connected;
\item $\CR\setminus A$ is locally connected at $\infty$.
\end{enumerate}
Suppose that $g\colon A\to\C$ is a continuous function that is holomorphic on $\interior(A)$. 
Then for every $\varepsilon>0$, there exists an entire function $f$ such that
$$
|f(z)-g(z)|<\varepsilon\quad \mbox{ for all } z\in A.
$$
\end{thm}

Here $\CR\setminus A$ is \emph{locally connected at $\infty$} if $\infty$ has 
 a neighbourhood base consisting of connected sets. Since $A$ is closed, this 
 is equivalent to the following: For every $R>0$ there is an $R'\geq R$ such that
 any point $z\in \CR\setminus A$ with $\lvert z\rvert \geq R'$
 can be connected to infinity by a curve in $\C \setminus A$ consisting of points of
 modulus greater than $R$ (see~\cite[Lemma on p.~138]{gaier87}). In particular, the conditions of Arakelyan's theorem
 are satisfied whenever $A= \bigcup_{k=0}^{\infty} A_k$, where the $A_k\subseteq \C$ are
  pairwise disjoint non-empty closed subsets such that 
 \begin{itemize}
  \item if $A_k$ is unbounded, then $A_k\cup\{\infty\}$ is locally connected;
  \item all connected components of $\C\setminus A_k$ are unbounded;
  \item the $A_k$ tend to infinity in the Hausdorff metric as $k\to\infty$; that is,
    if $(z_k)_{k=0}^{\infty}$ is a sequence with $z_k\in A_k$ for all $k$,
    then $z_k\to\infty$ as $k\to \infty$. 
 \end{itemize}
  In this article, we will only apply Arakelyan's theorem to sets of this form;
   in fact each $A_k$ will be either a full compact set,  
    a topological half-strip (bounded by a single arc tending to infinity in both directions) or a topological strip (bounded by two arcs tending to
    infinity in both directions).

In each of our constructions, we use two simple facts concerning
  approximation. The first is essentially a uniform version of Hurwitz's theorem, while the second
  is an elementary exercise. For the reader's convenience, we include the
  proofs.
 \begin{lem}[Approximation of univalent functions]\label{lem:approx1}
   Let $U,V\subseteq\C$ be open, and let $\phi\colon U\to V$ be a conformal isomorphism.
   
Let $A\subseteq U$ be a closed set such that $\dist(A,\partial U)>0$ and 
  $\eta \defeq \inf_{z\in A}\lvert \phi'(z)\rvert>0$.\linebreak Then there is $\eps>0$ with the following property: if $f\colon U\to\C$ is holomorphic with $\lvert f(z) - \phi(z)\rvert \leq\eps$ for all $z\in U$, then 
      $f$ is injective on $A$, with $f(A)\subseteq V$. 
      
   Moreover, $\lvert f'(z) - \phi'(z)\rvert < \eta/2$ for $z\in A$. In
     particular, $\lvert f'(z)\rvert > \eta/2$, and if $\lvert \phi'(z)\rvert$ is bounded from above on $A$, then so is
     $\lvert f'(z)\rvert$. 
 \end{lem} 
\begin{rmk}
  When $A\subseteq U$ is compact, the hypotheses on $A$ are automatically satisfied.
\end{rmk}
 \begin{proof}
   Let $\delta <\dist(A,\partial U)/2$. By Koebe's $1/4$-theorem, 
      \[ \phi(\D(z_0,\delta))\supseteq \D(\phi(z_0), \rho) \]
    for all $z_0\in A$,  where $\rho \defeq \delta\cdot\eta/4$. In particular, 
      \begin{equation}\label{eqn:distphiA} \dist(\phi(A),\partial V)\geq \rho >0 \end{equation}
     and, since $\phi$ is injective,
         \begin{equation}\label{eqn:phiestimate} \lvert \phi(z)-\phi(z_0)\rvert \geq \rho \end{equation}
      for all $z_0\in A$ and $z\in U$ with $\lvert z - z_0\rvert \geq \delta$. 
    
   Now let $f$ be as in the statement of the lemma, where 
      $\eps < \rho/2$.
      Then 
      \begin{equation}\label{eqn:largescaleinjective}
           f(z)\neq f(z_0) \quad\text{when }\lvert z - z_0\rvert \geq \delta, \end{equation}
     by~\eqref{eqn:phiestimate}. Moreover, $f(A)\subseteq V$ by~\eqref{eqn:distphiA}. 
     
     We must show that also $f(z)\neq f(z_0)$ for $z\neq z_0$ when $z\in D\defeq \D(z_0,\delta)$. In other words, we claim that $f(z)-f(z_0)=0$ has a unique
     solution $z\in D$. According to the argument principle, the number of such solutions is given
       by the winding number of $f(\partial D)$ around $f(z_0)$. By choice of $\eps$ and~\eqref{eqn:phiestimate}, the curves
       $\phi(\partial D)-\phi(z_0)$ and $f(\partial D)- f(z_0)$ are homotopic in $\C\setminus \{0\}$. Thus, they have the same winding number around $0$. As $\phi$ is injective,
       that winding number is $1$, and the claim is proved. Together with~\eqref{eqn:largescaleinjective}, we see that $f$ is injective on $A$. 

      Finally, by Cauchy's theorem, we have 
     \begin{equation}\label{eqn:derivativeapprox} \lvert f'(z_0) - \phi'(z_0)\rvert = 
        \frac{1}{2\pi}\left\lvert \int_{\partial \D(z_0,\delta)} \frac{f(\zeta)-\phi(\zeta)}{(\zeta - z_0)^2}\,
         \textrm{d} \zeta\right\rvert \leq \frac{2\pi\delta}{2\pi}\frac{\eps}{\delta^2}=\frac{\eps}{\delta} < \frac{\eta}{2} \end{equation}
     for $z_0\in A$. 
        This proves the final statement of the lemma.
 \end{proof}

\begin{lem}[Approximation of iterates]\label{lem:approx2}
 Let $U\subseteq\C$ be open, and let $g\colon U\to\C$ be continuous.
   Suppose that 
    $K\subseteq U$ is closed and $n\in \N$ is such that $g^k(K)$ is defined and a subset of $U$
     for all $k < n$. 
     
    Suppose furthermore that $\widehat{K} \defeq \bigcup_{k=0}^{n-1} g^k(K)$ satisfies $\dist(\widehat{K},\partial U)>0$ and 
      that $g$ is uniformly continuous at every point of $\widehat{K}$, with respect to Euclidean distance.
    
    Then for every $\eps>0$, there is $\delta>0$ 
     with the following property. If $f\colon U\to\C$ is continuous
     with $\lvert f(z) - g(z)\rvert < \delta$ for all $z\in U$, then 
       \begin{equation}\label{eqn:iteratesclose} \lvert f^k(z) - g^k(z)\rvert < \eps   
       \end{equation}
       for all $z\in K$ and all $k\leq n$.  
\end{lem}
\begin{rmk}\label{rmk:uniformcontinuity}
 By the uniform continuity assumption in the second paragraph, we mean that 
   there exists a modulus of continuity that is valid at all points of $\widehat{K}$. That is, 
   for every $\eps>0$ there is $\delta'(\varepsilon)<\dist(\widehat{K},\partial U)$ 
   such that $\lvert {g}(\zeta) - {g}(\omega)\rvert < \eps/2$ whenever ${\omega} \in \widehat{K}$ and 
    $\zeta\in U$ with $\lvert \zeta  - \omega\rvert \leq \delta'(\varepsilon)$. This assumption is satisfied, in particular,
    if $g$ is uniformly continuous on a neighbourhood $V\subseteq U$ of $\widehat{K}$ with $\dist(\partial V, \widehat{K})>0$,
    or if $g$ is holomorphic and $\lvert g'\rvert$ is uniformly bounded above on such a neighbourhood $V$. 
    
  In particular, the assumptions on $\widehat{K}$ are automatically satisfied when $K$ is compact. 
\end{rmk}
\begin{proof}
  Fix $U$, $g$ and $K$ as in the statement of the lemma. For every $n\in\N$ that satisfies the hypotheses, we prove the existence of a function $\delta_n$ 
   such that $\delta = \delta_n(\eps)$ has the desired property. The proof
   proceeds by induction on $n$; for $n=1$ we may set 
     $\delta_1(\eps)=\eps$. Suppose that the induction hypothesis holds
    for $n$, and let $\eps>0$. Define
\[
    \delta_{n+1}(\varepsilon)=\min\bigl\{\delta_n(\eps),\delta_n(\delta'(\eps)), \eps/2\bigr\}
    \] 
    where $\delta'(\eps)$ is as in Remark~\ref{rmk:uniformcontinuity}.
    Let $z\in K$. Then~\eqref{eqn:iteratesclose} holds for $k\leq n$ by
    the induction hypothesis. Setting $\zeta\defeq {f}^n(z)$ and $\omega \defeq {g}^n({z})$,
     we have $\lvert \zeta - \omega\rvert \leq \delta'(\eps)$ by the induction hypothesis,
     and, in particular, $\zeta\in U$. Thus, 
\[   \lvert f^{n+1}(z)-g^{n+1}(z) \rvert \leq |f(\zeta)-g(\zeta)|+|g(\zeta)-g(\omega)|<\varepsilon, \] 
    as required.
\end{proof} 

\begin{cor}[Approximating univalent iterates]\label{cor:approx}
 Let $U\subseteq\C$ be open and $g\colon U\to\C$ be holomorphic. Suppose that $G\subseteq U$ is open and 
   $K\subseteq G$ is closed with the following properties for some $n\geq 1$:
   \begin{enumerate}[(a)]
     \item $g^n$ is defined and univalent on $G$;
     \item $\lvert g'\rvert$ is bounded from above and below by positive constants on $U$;\label{item:derivativeboundedonU}
     \item $\dist(K,\partial G)>0$.\label{item:distancetopartialG}
   \end{enumerate}
   Then for every $\eps>0$, there is $\delta>0$ 
     with the following property.
   For any holomorphic $f\colon U\to\C$ with $\lvert f(z) - g(z)\rvert \leq \delta$ for all
     $z\in U$, $f^n$ is defined and injective on $K$, $\lvert f^k(z) - g^k(z)\rvert \leq \eps$
     on $K$ for $k\leq n$, 
     $\lvert f'\rvert$ is bounded from above and below by positive constants on 
     $\widehat{K} = \bigcup_{k=0}^{n-1} f^k(K)$, and $f$ is uniformly continuous at every point
     of $\widehat{K}$ in the sense of Remark~\ref{rmk:uniformcontinuity}.
\end{cor}
\begin{rmk}
  If $K$ is compact, then hypotheses~\ref{item:derivativeboundedonU} and~\ref{item:distancetopartialG} can be omitted. Indeed, in this 
  case~\ref{item:distancetopartialG} is trivial and~\ref{item:derivativeboundedonU} 
     is automatically satisfied for the restriction of $f$ to 
      a neighbourhood of $\widehat{K}$. 
\end{rmk}
\begin{proof}
 Note that $g^k$ is univalent on $G$ for $1\leq k \leq n$. 
    Furthermore, by Koebe's theorem and the assumption on $g'$, 
    $g$ is uniformly
   continuous at every point of $g^{k-1}(K)$ and 
   $\dist(g^k(K),\partial U) > 0$, for all $k\geq 1$. 
    So $\dist(\widehat{K},\partial U)>0$ and the hypotheses of
   Lemma~\ref{lem:approx2} are satisfied. Thus, there is $\delta>0$ such that 
   $\lvert f^k(z) - g^k(z)\rvert \leq \eps$ for all $z\in K$ and all $k\leq n$ 
   if $\lvert f(z)-g(z)\rvert \leq \delta$ on $U$. 
   
 If $\eps$ is chosen small enough, then for $k\leq n$ 
   the hypotheses of Lemma~\ref{lem:approx1} 
   are satisfied for $\phi = g^k$ on $G$. It follows that $f^k$ is injective 
   on $K$ with $\lvert (f^k)'\rvert$ bounded above and below by positive constants. 
   Since this holds for all $k\leq n$, we see that $\lvert f'\rvert$ is bounded from above and below by
   positive constants on each $f^k(K)$. This completes the proof of the lemma, apart from the statement about uniform continuity. 

  For that final statement, observe that for $\eta>0$, 
     the set $K(\eta) \defeq \{z\in \C\colon \dist(z,K)\leq \eta\}$ also satisfies the 
     hypotheses of the corollary, if $\eta < \dist(K,\partial G)$. Now we  apply the 
     part of the corollary that we just proved to $K(\eta)$, and conclude that the derivative of $f$ is bounded from above on 
     $\widehat{K}(\eta) = \bigcup_{k=0}^{n-1} f^k(K(\eta))$. By~\ref{item:derivativeboundedonU}, 
     $\dist(\partial \widehat{K}(\eta),\widehat{K})>0$. As noted in Remark~\ref{rmk:uniformcontinuity}, this implies that $f$ is uniformly continuous at every point
     of $\widehat{K}$, as desired.
\end{proof}

In Sections~\ref{sec:uniform} and~\ref{sec:maverickconstruction}, we also
  require the following fact about approximating compact and full sets from
  above by finite collections of Jordan domains; see Figure~\ref{fig:rabbit}.
    This is a classical fact of plane topology,
    used already by Runge to prove Theorem~\ref{thm:runge}~\cite[pp.~230--231]{runge85} (see also \cite[Theorem~2 on pp.~7--8]{walsh69}).
    For the reader's convenience, we provide a simple proof.
  
  \begin{figure} 
\includegraphics[width=.55\linewidth]{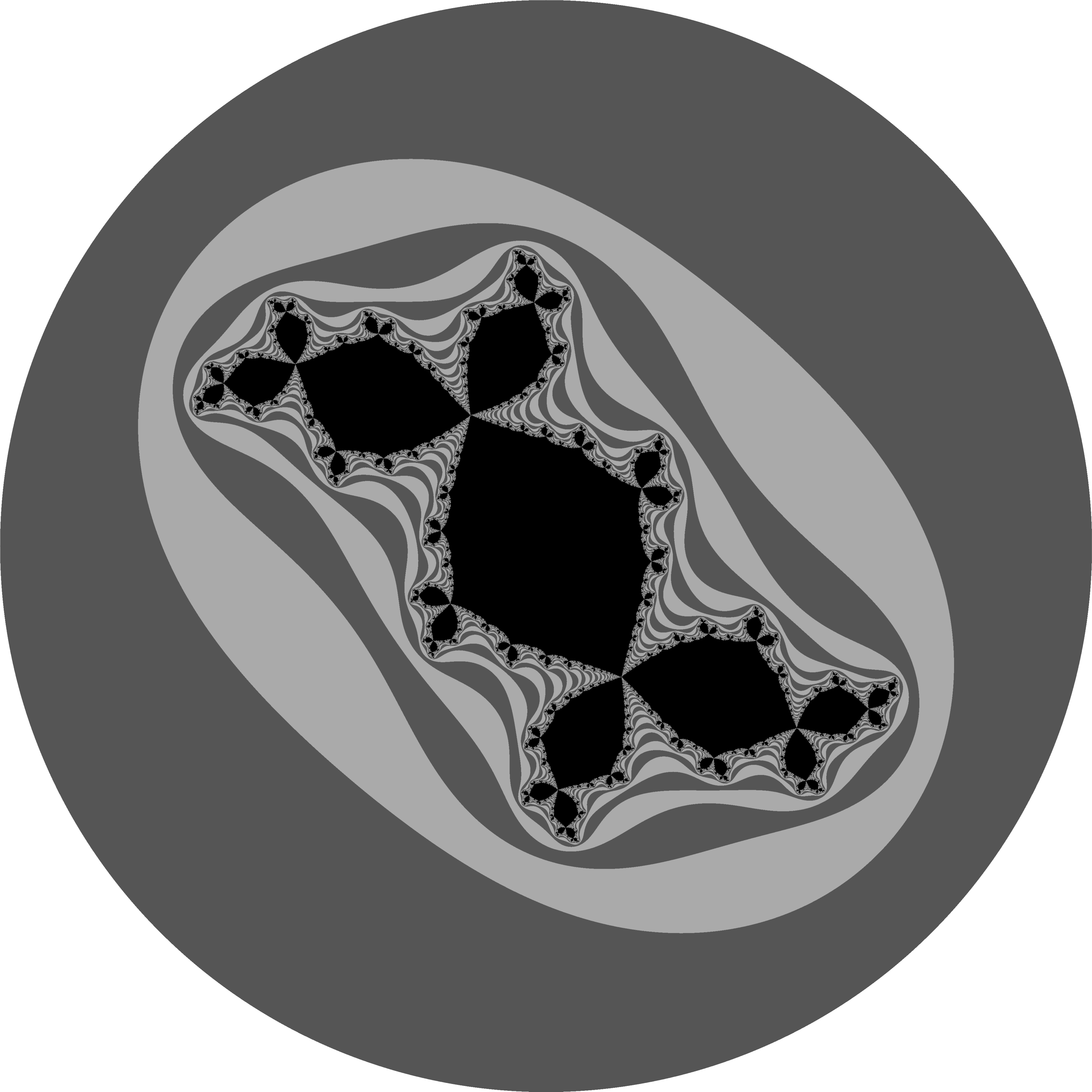} 
\caption{Nested sequences of compacta $(K_j)$ shrinking down to the filled-in Julia set $K = \K(p)$ of  $p(z)=z^2+c$ with $c\approx -0.12+0.74i$.}
\label{fig:rabbit}
\end{figure}

  \begin{lem}[Approximation by unions of Jordan domains]\label{lem:Kn}
   Let $K\subseteq \C$ be compact and full. Then there exists a sequence $(K_j)_{j=0}^{\infty}$
     of compact and full sets
    such that 
     $K_{j}\subseteq \interior(K_{j-1})$ for all $j\in\N$ and 
         \[ \bigcap_{j=0}^\infty K_j = K. \]
    Each $K_j$ may be chosen to be bounded by a finite disjoint union of closed Jordan curves. 
 \end{lem} 
 \begin{proof}
   Define 
    \[
     \tilde{K}_j:=\textup{fill}\bigl(\bigl\{z\in \C \colon \textup{dist}(z,K)\le \tfrac{1}{j}\bigl\}\bigl). \]
    Each $\tilde{K}_j$ is compact and full by definition, and clearly $K\subseteq \tilde{K}_{j+1} \subseteq \interior(\tilde{K}_j)$ for all $j$.
    Any $z\in \C\setminus K$ can be connected to $\infty$ by a curve $\gamma$ disjoint from $K$,
       and hence we have $z\notin \tilde{K_j}$ for all sufficiently large $j\in \N$.  
       
    Now fix $j$ and let $V_1,\dots,V_m$ be the finitely many connected components of $\interior(\tilde{K}_j)$ that intersect
      $K$. Then each $V_{\ell}$ is a simply connected domain, and hence (say by the Riemann mapping theorem) there is 
      a Jordan domain $U_{\ell}\subseteq V_{\ell}$ with $K\cap V_{\ell}\subseteq U_{\ell}$. The sets
       $K_j \defeq \bigcup_{\ell=1}^m \overline{U_{\ell}}$ then have the required property. 
 \end{proof}

\section{Uniform escape}

\label{sec:uniform}

In this section, we prove Theorem~\ref{thm:main} when $Z_{\BU}=\emptyset$,
  in which case we can choose $f$ such that the iterates converge to
  infinity uniformly on all of $K$. 

\begin{thm}[Wandering compacta with uniform escape]\label{thm:uniform}
Let $K\subseteq \C$ be a full compact set. Then there exists a transcendental entire function $f$ such that 
 \begin{enumerate}[(i)]
     \item  $f^n(K)\cap f^m(K)=\emptyset$ when $n\neq m$;\label{item:uniformdisjointness}
    \item ${f^n}\vert_K \to \infty$ uniformly as $n\to\infty$.\label{item:uniformescape}
    \item $\partial K\subseteq \J(f)$;\label{item:uniformboundaryinJ}
     \item  every connected component of $\interior(K)$ is a wandering domain of $f$;\label{item:uniformwandering}
   \end{enumerate}
\end{thm} 

This result was proved by
  Boc Thaler~\cite{bocthaler21} in the case where $K$ is the closure of a simply connected
  domain $U$. 
  Apart from a change of perspective~-- starting with the compact set $K$ instead of the
  domain $U$~-- our proof follows similar lines, but with some 
  modifications that will become important later. For example,~\cite{bocthaler21} uses a stronger version of Runge's theorem, due to Eremenko and Lyubich, in which
  the approximating function is required to agree with the original one at finitely many given points. This allows Boc Thaler
  to prescribe exactly the orbits of a sequence of points accumulating on the boundary $\partial U$, 
  ensuring that this boundary is contained in the Julia set. We instead use the original, unmodified version of Runge's theorem, which
  yields less control over the exact orbits, but now allows us to ensure that the union of a sequence of potentially uncountable sets (the sets $P_j$ below), accumulating on
  $\partial K$, belongs to a basin of attraction. While this does not play an essential  role in the proof of Theorem~\ref{thm:uniform} (other than to simplify it), it
  is of crucial importance in the
  proof of Theorem~\ref{thm:eremenkocounterexample} in Section~\ref{sec:eremenko}, where the $P_j$ are replaced by a sequence of unbounded connected sets that 
  will separate the desired point component from other points in the escaping set. This approach is also essential for Theorem~\ref{thm:strongeremenko} below, which provides 
   simple new counterexamples to the
  strong Eremenko conjecture.

   To set up the proof of Theorem~\ref{thm:uniform}, let $K\subseteq\C$ be a full compact set and choose
      a sequence $(K_j)$ of approximating sets according to Lemma~\ref{lem:Kn}. By applying an affine
      transformation, we may assume without loss of generality that $K_0\subseteq \DD$. 
      For $j\geq 0$, also choose a non-separating compact and $2^{-j}$-dense subset
         \[ P_j \subseteq \partial K_j. \]
      That is, $\dist(z,P_j)\leq 2^{-j}$ for all $z\in \partial K_j$; in particular, $\partial K$ is the Hausdorff limit of the sets $P_j$. 
        We shall construct a 
        function $f$ such that all $P_j$ are contained in a basin of attraction, while $f^n$ tends to infinity on $K$ itself; 
        this ensures that $\partial K\subseteq \J(f)$. In most of our applications,
        we may choose the $P_j$ as finite sets, but for the proof
        of Theorem~\ref{thm:strongeremenko} below we shall use larger sets $P_j$. 

For $j\geq -1$, consider the discs
    \[ D_j \defeq \D(3j,1)=\big\{z\in\C\colon |z-3j|<1\big\}. \] 
         Our main goal now is to prove the following proposition, which implies Theorem~\ref{thm:uniform}.

     \begin{prop}\label{prop:main}
       There exists a transcendental entire function $f$ with the following properties:
        \begin{enumerate}[(a)]
          \item $f(\overline{D_{-1}})\subseteq D_{-1}$;\label{item:discinvariant}
          \item $f^{j+1}(P_{j})\subseteq D_{-1}$ for all $j\geq 0$;\label{item:Pjindisc}
          \item $f^j$ is injective on $K_j$ for all $j\geq 0$, with
               $f^j(K_{j})\subseteq D_{j}$.\label{item:Kjindisc}
        \end{enumerate}
     \end{prop}

 \begin{proof}[Proof of Theorem~\ref{thm:uniform}, using Proposition~\ref{prop:main}]
      Let $j\geq 0$; then $K\subseteq K_j$ by choice of $K_j$ and hence $f^j(K)\subseteq D_{j}$ by~\ref{item:Kjindisc}. This establishes~\ref{item:uniformdisjointness} and~\ref{item:uniformescape}.
      
       On the other hand, by~\ref{item:discinvariant} and~\ref{item:Pjindisc}, 
      we have  $f^k(P_{j})\subseteq D_{-1}$ for $k > j$. Since every point $z\in \partial K\subseteq K$ is the limit of 
       a sequence of points $p_j\in P_j$, it follows that the family $(f^k)_{k=1}^{\infty}$ is not equicontinuous at $z$, and thus 
       $z\in \J(f)$. This proves~\ref{item:uniformboundaryinJ}. 
       If $U$ is a connected component of $\interior(K)$, then $U\subseteq \F(f)$ 
       by~\ref{item:uniformescape} and the definition of the Fatou set, while
       $\partial U\subseteq \J(f)$ by~\ref{item:uniformboundaryinJ}. So $U$ is a 
       Fatou component, and it is wandering by~\ref{item:uniformdisjointness}.
 \end{proof}

\begin{proof}[Proof of Proposition~\ref{prop:main}]
  We construct $f$ as the limit of a sequence of polynomials
    $(f_j)_{j=0}^{\infty}$,
   which are defined inductively using Runge's theorem. More precisely, 
   for $j\geq 1$, the function $f_j$ approximates a function $g_j$, defined and holomorphic
   on a neighbourhood of a compact set $A_j\subseteq\C$, 
   up to an error of at most $\eps_j > 0$. The function $g_j$
   in turn is defined in terms of the previous function $f_{j-1}$.

   Define 
     \[ \Delta_j       \defeq \overline{\D(-3,1+3j)} \supseteq D_{j-1}\]
    for $j\geq 0$. 
   The inductive construction ensures the following properties:
    \begin{enumerate}[(i)]
      \item For every $j\geq 0$, 
         $f_{j}^j$ is injective on $K_j$ and $f_{j}^j(K_j)\subseteq D_j$.\label{item:injectivity}
      \item For $j\geq 1$, $\Delta_{j-1} \subseteq A_{j}\subseteq \Delta_{j}$.\label{item:nested}
      \item $\eps_1 < 1/2$ and $\eps_{j} \leq \eps_{j-1}/2$ for $j\geq 2$.\label{item:inductiveclose}
    \end{enumerate}

  To anchor the induction, we set
     $f_{0}(z) \defeq -3$ for $z\in\C$ and note that~\ref{item:injectivity} 
     holds trivially for $j=0$. 

   Let $j\geqslant 0$ and suppose that $f_{j}$ has been defined, and that $\eps_{j}$ has been defined if $j\geq 1$. 
      Applying Lemma~\ref{lem:Kn} to $K_{j+1}$, we find a full compact set
        $L_{j}\subseteq K_{j}$ with $K_{j+1}\subseteq \interior(L_{j})$. 
      Set $Q_{j} \defeq f_{j}^{j}(P_{j})$. 
      By~\ref{item:injectivity},  
        $Q_{j}\subseteq D_{j}$ and 
        $Q_{j} \cap (\partial D_{j} \cup f_{j}^{j}(L_{j})) = \emptyset$. 
       Let $R_{j}$ be a compact full neighbourhood of $Q_{j}$ disjoint from
         $\partial D_{j}$ and $f_{j}^{j}(L_{j})$, and set
         \[ A_{j+1} \defeq \Delta_{j} \cup R_{j} \cup f_{j}^{j}(L_{j}). \] 
       We have $R_{j}\subseteq D_{j}$ 
          and $f_{j}^{j}(L_{j})\subseteq f_{j}^{j}(K_{j})\subseteq D_{j}$ by the inductive hypothesis~\ref{item:injectivity}. In particular, neither set intersects 
        $\Delta_{j}$. 
        It follows that $A_{j+1}$ satisfies~\ref{item:nested}
        and the hypotheses 
         of Runge's theorem. 
      We define
         \[ g_{j+1} \colon A_{j+1} \to \C;\quad z\mapsto
             \begin{cases}
                 f_{j}(z), & \text{if } z\in \Delta_{j}, \\
                 - 3, &\text{if }z\in R_{j}, \\ 
                 z+3, &\text{if } z\in f^{j}_{j}(L_{j}). \end{cases} \] 
      (See Figure~\ref{fig:proofsketch}.) By definition, the function $g_{j+1}$ extends analytically 
        to a neighbourhood of $A_{j+1}$. 
       Observe that $g_{j+1}^{j+1}(P_j) = g_{j+1}(Q_j) \subseteq g_{j+1}(R_j) =\{-3\} \subseteq D_{-1}$, and that 
        $g_{j+1}^{j+1}$ is defined and univalent on $\interior(L_{j})$ by~\ref{item:injectivity}.

\begin{figure}
\begin{center}
\def\svgwidth{\textwidth}
\input{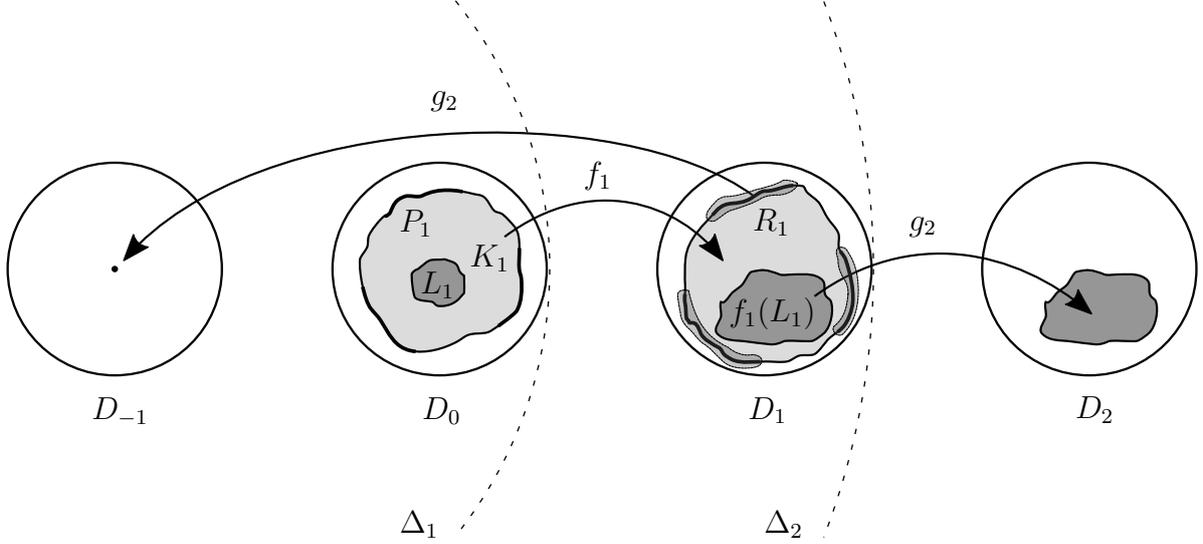} 
\end{center}
\vspace*{-20pt}
\caption{The construction of $g_j$ in the proof of Proposition~\ref{prop:main}, for $j=2$.}\label{fig:proofsketch}
\end{figure}
       
     Now choose $\eps_{j+1}$ according to~\ref{item:inductiveclose}
       and sufficiently small that  
       any entire function~$f$ with 
         $\lvert f(z) - g_{j+1}(z)\rvert \leq 2\eps_{j+1}$ on $A_{j+1}$ satisfies: 
       \begin{enumerate}[(1)]
          \item $f^{j+1}(P_{j})\subseteq D_{-1}$;\label{item:pjclose}
          \item $f^{j+1}$ is injective on $K_{j+1}\subseteq \interior(L_{j})$;\label{item:fjinjective}
          \item $f^{j+1}(K_{j+1}) \subseteq D_{j+1}$.\label{item:fjimage} 
       \end{enumerate}
       Here~\ref{item:pjclose} is 
         possible by Lemma~\ref{lem:approx2},
           and~\ref{item:fjinjective} and~\ref{item:fjimage} are possible by 
           Corollary~\ref{cor:approx}. 
       We now let $f_{j+1}\colon \C\to\C$ be
       a polynomial approximating $g_{j+1}$ up to an error of at most $\eps_{j+1}$, 
       according to Lemma~\ref{thm:runge}. This completes the inductive 
       construction. 
       
    Condition~\ref{item:inductiveclose} implies that $(f_j)_{j=k}^{\infty}$ forms
     a Cauchy sequence on every set $A_k$, and by~\ref{item:nested}, $\bigcup_{k=1}^\infty A_k=\C$. So the functions $f_j$
     converge locally uniformly 
     to an entire function~$f$. For $1\leq k\leq j$,
     $$
     |f_j(z)-g_k(z)|\leq \eps_k+ \dots + \eps_j \leq 2\eps_k\quad \textup{for all } z\in A_k.
     $$ 
     Hence the limit function $f$ satisfies
     $\lvert f(z) - g_k(z)\rvert\leq 2\eps_k$ for all $z\in A_k$ and $k\geq 1$. 

    Since $g_{1}(D_{-1})=f_{0}(D_{-1})=\{-3\}$, and 
       $2\eps_1 < 1$, it follows that
         $f(\overline{D_{-1}})\subseteq D_{-1}$.  
        Moreover, 
    $f^{j+1}(P_j)\subseteq D_{-1}$ for $j\geq 0$ by~\ref{item:pjclose}. Finally, by~\ref{item:fjinjective} and~\ref{item:fjimage}, 
    $f^j$ is injective on $K_j$ and $f^j(K_j)\subseteq D_j$. This completes the proof of Proposition~\ref{prop:main}.
\end{proof}

We now prove Theorem~\ref{thm:fatou} regarding the existence of Fatou components with a common boundary.

\begin{proof}[Proof of Theorem~\ref{thm:fatou}]
Let $X\subseteq \C$ be  a continuum
 such that $\C\setminus X$ has infinitely many connected components and 
 the boundary of every such component coincides with $X$. As mentioned
 in the introduction, such a continuum was first constructed by
 Brouwer~\cite[p.~427]{brouwer10}, and can be obtained using
 the construction described by Yoneyama~\cite[p.~60]{yoneyama17}.
Let $U$ be the unbounded connected component of $\C\setminus X$. 
 
 The set $K\defeq \Fill(X)=\C\setminus U$ satisfies $\partial K = \partial U = X$. 
  Apply Theorem~\ref{thm:main} to the full continuum $K$ to obtain 
  a transcendental entire  
  function $f$ for which 
  $X = \partial K\subseteq \J(f)$, and every connected component of
  $\interior(K)\subseteq \F(f)$ is a simply connected wandering domain. Since there
  are infinitely many such components, each of which is bounded by~$X$, Theorem~\ref{thm:fatou} is proved.
\end{proof}

Note that the set $K$ constructed in the proof of Theorem~\ref{thm:uniform} escapes to infinity rather slowly, but 
  we can easily modify the construction to increase the rate of escape, by replacing the discs $D_j$ by any other disjoint sequence of discs tending to infinity
  (and the map $z+3$ used in the definition of $g_{j+1}$ by an affine map taking $D_j$ to $D_{j+1}$). Moreover, the disc $D_{j+1}$ need not be chosen in advance; it may be chosen
  to depend on the map $f_j$. In this manner, we can ensure that $K$ belongs to the \emph{fast escaping set} 
  \begin{equation}\label{eqn:fastescapingset}
     A(f) \defeq \{z\in I(f)\colon \text{there is}~ n_0\geq 0~  \text{such that}~ \lvert f^{n_0+j}(z) \rvert \geq M^j(r,f) ~\text{for all}~ j\geq 0 \}. \end{equation}
  Here $r$ is any number sufficiently large that $D(0,r)\cap J(f) \neq \emptyset$, and $M^j(r,f)$ denotes the $j$-th iterate of the maximum modulus function 
   $r\mapsto M(r,f)\defeq \max_{\lvert z\rvert \leq r} \lvert f(z)\rvert$.  
 \begin{prop}\label{prop:fastescaping}
    The function $f$ in Theorem~\ref{thm:uniform} can be chosen such that $K\subseteq A(f)$.
 \end{prop} 
 \begin{proof}
   We modify the inductive construction of $g_{j+1}$ by choosing both the discs $\Delta_j$ and $D_j$ inductively during the construction. Each $\Delta_j$ is a closed disc of
   radius $r_j$ 
   centred at the origin, where $r_j$ is a rapidly increasing sequence, 
   while all $D_j$ are open discs of radius $1$. As before, they
   satisfy $D_j\subseteq \Delta_{j+1}$ and $D_j\cap \Delta_j=\emptyset$.
     The initial choices of $D_{-1}$, $D_0$ and $\Delta_0$ remain unchanged. 
   The radius $r_1$ is chosen large enough to ensure that $\Delta_0\cup D_0 \subseteq \Delta_1$.
     
    Prior to defining $g_{j+1}$, we choose $r_{j+1}$ (if $j>0$) greater than $M(r_j,f_j)+2$, and then choose the disc $D_{j+1}$ to be a disc of radius $1$ that is disjoint from $\Delta_{j+1}$. This construction ensures that $D_j\subseteq \Delta_{j+1}$: this holds by assumption for $j=0$, and follows inductively since
      $D_{j+1}$ contains the image of a point in $D_j$, which has modulus less than $M(r_j,f_j)<r_{j+1}-2$. 
      
   The remainder of the construction
     remains unchanged, except that on $f_j^j(L_j)$, we define $g_{j+1}$ to be the restriction of an affine function mapping $D_j$ to $D_{j+1}$. 

   Let $f$ be the resulting function. Since $\lvert f - f_j\rvert < 2\eps_j<2$ on $\Delta_j$, we see that $r_{j+1} > M(r_j,f_j)+2 > M(r_j,f)$ for all $j\geq 1$. In particular,
      $r_j > M^{j-1}(r_1,f)$. If $z\in K$, then $f^{j+1}(z) \in D_{j+1}$, and hence $\lvert f^{j+1}(z)\rvert > r_{j+1} > M^j(r_1,f)$ for all $j\geq 0$. Since 
      $D(0,r_1)$ contains $K$, and hence intersects the Julia set, it follows that $K\subseteq A(f)$, as claimed. 
 \end{proof}

As mentioned in the introduction, we can use Proposition~\ref{prop:main} to give new counterexamples to the strong version of Eremenko's conjecture.

\begin{thm}[Counterexamples to the strong Eremenko conjecture]
\label{thm:strongeremenko} 
Let $X\subseteq \C$~be a full continuum. Then there exists a transcendental entire function~$f$ such that 
every path-connected component of $X$ is a path-connected component of the escaping set $\I(f)$, and every path-connected component of $\partial X$ is a 
path-connected component of $\J(f)$. In particular, no point of $X$ can be connected to $\infty$ by a curve in $\I(f)$.
\end{thm}
\begin{proof}
If $X$ is a singleton, then we may consider a continuum 
 $X'$ of which $X$ is a path-connected component and continue the proof with $X'$ instead of $X$. 
 (For example, we can take $X'$ to be as shown in Figure~\ref{fig:point-path-connected-component}.)
 From now on we therefore suppose that $X$ is not a singleton.
\begin{figure}
\includegraphics[width=.7\linewidth]{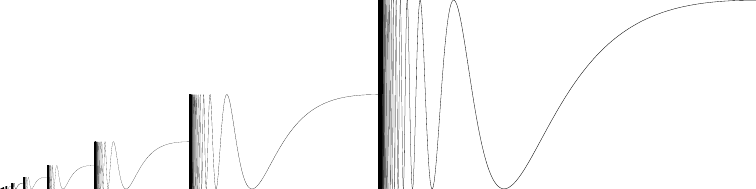} 
\caption{A continuum $X'$ having a singleton path-connected component. $X'$ is obtained as a countable union of progressively smaller copies of the $\sin(1/x)$ continuum,
 accumulating on a single point (the left-most point in the figure).}
\label{fig:point-path-connected-component}
\end{figure}

Construct a function $f$ as in Proposition~\ref{prop:main}, where the set 
  $P_j\subseteq \partial K_j$ 
  is chosen such that no point of $K$ is accessible from
    $\C\setminus (K\cup \bigcup_{j=0}^{\infty} P_j)$. 
   For example, let $z_0,z_1$ be distinct points of $\partial K$, 
   and for each $j$ choose an open arc $\gamma_j\subseteq \partial K_j$
   in such a way that $\diam(\gamma_j)\leq 2^{-j}$ for all $j$ and
   such that, furthermore, $\gamma_{2j}\to z_0$ and $\gamma_{2j+1}\to z_1$ as $j\to\infty$. 
   Then the sets $P_j\defeq \partial K_j\setminus \gamma_j$ have the desired property.
   
  Now apply Proposition~\ref{prop:main}, to obtain a function $f$ satisfying
    the conclusions of Theorem~\ref{thm:uniform}. Each $P_j$ is contained in
    an iterated preimage of the forward-invariant disc $D_{-1}$, and hence is
    disjoint from $\J(f)\cup \I(f)$. By the choice of $P_j$, this implies that 
    there is no curve in $\J(f)\cup \I(f)$ connecting a point of $\C\setminus K$ to
    a point of $K$. Since $K\subseteq \I(f)$, every path-connected component of $K$
    is a path-connected component of $\I(f)$; similarly, since $\partial K = K\cap \J(f)$, 
    every path-connected component of $\partial K$ is a path-connected component
    of $\J(f)$. 
\end{proof}
\begin{rmk}
  As in Proposition~\ref{prop:fastescaping}, we can again choose $f$ so that $X\subseteq A(f)$.
\end{rmk}

\section{Scaffolding}

\label{sec:scaffolding}

The functions whose existence is asserted in Theorems~\ref{thm:eremenkocounterexample} and~\ref{thm:main} will both be constructed using a 
  sequence of unbounded horizontal strips $(S_j)_{j=0}^{\infty}$ and $(T_j)_{j=0}^{\infty}$, such that $S_j$ is mapped
  univalently over $S_{j+1}$ and $T_{j+1}$. (See Figure~\ref{fig:strips}.) 
  Our wandering set $K$ 
  starts out in the strip $T_0$, maps into $S_0$ and from there into $T_1$, then back
  into $S_0$ and on to $T_2$ via $S_1$, and so on. The crucial step
  happens at the time     
    \begin{equation}\label{eqn:Nj} N_j \defeq \sum_{\ell =1}^j (\ell+1) = \frac{j\cdot (j+3)}{2}, \end{equation}  
    when 
  $f^{N_j}(K)$ is in $T_j$, and is mapped back inside $S_0$ by $f$. (By convention, $N_0\defeq 0$.) 

\begin{figure}
\begin{center}
\def\svgwidth{.9\textwidth}
\input{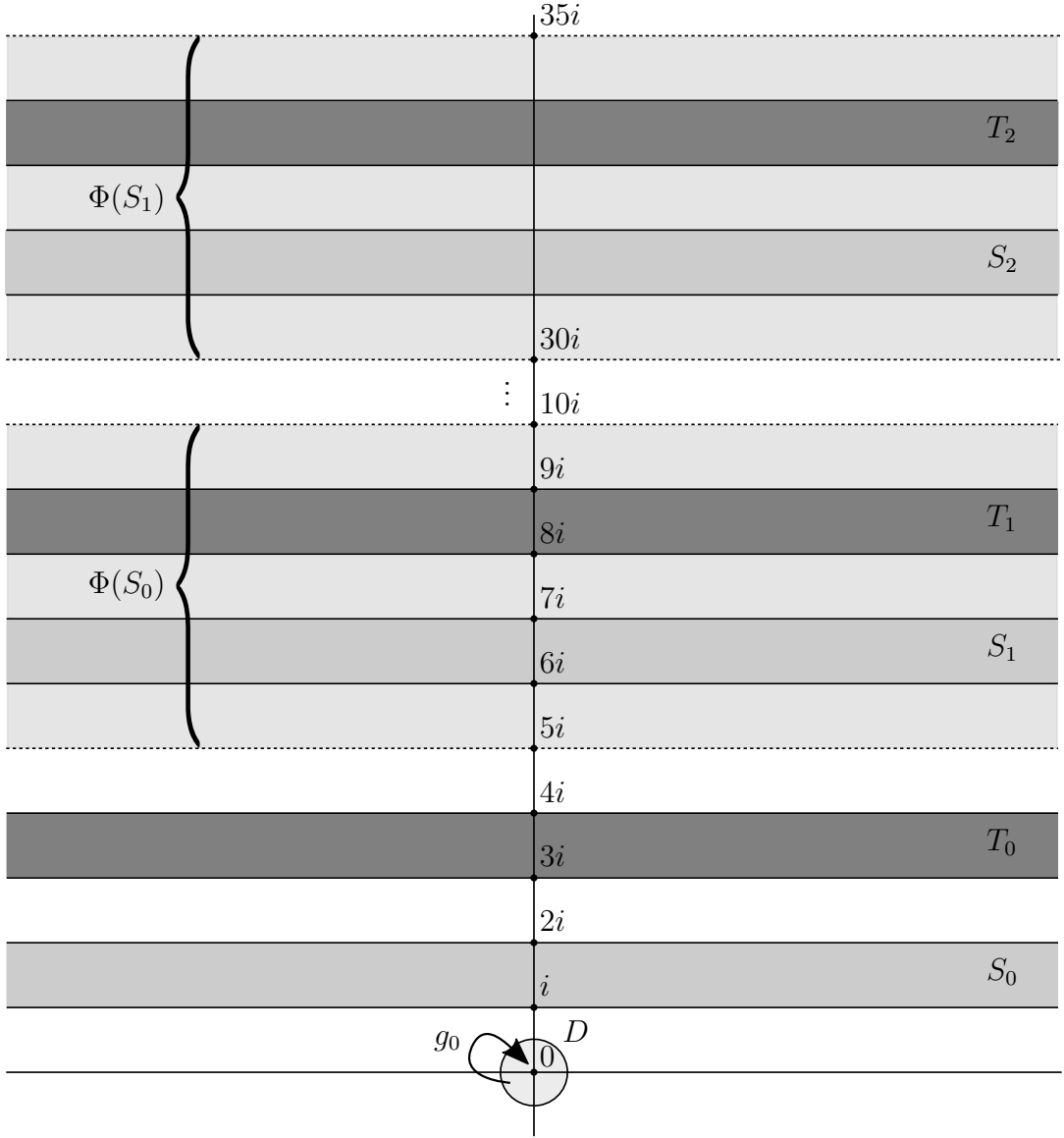}
\end{center}
\vspace*{-5pt}
 \caption{\label{fig:strips}The strips $S_k$ and $T_k$ used in the construction.}
\end{figure}

Similarly as in the proof of Theorem~\ref{thm:uniform},
  $f$ is the limit of a sequence of functions $(f_j)_{j=0}^{\infty}$ constructed using approximation theory. 
  We have to ensure that all these functions share the 
  mapping behaviour on the strips $(S_j)$ described above. This can be achieved
  using Lemma~\ref{lem:approx1} together with Arakelyan's theorem. 
  
  To provide the details of this set-up, consider the affine map
   \[  \Phi\colon\C\to\C; \quad z\mapsto 5z. \]
   For $j\geq 0$, set
\begin{align*}
 S_j &\defeq \big\{z\in\C\colon   5^{j+1}-1 < 4 \im z  < 5^{j+1}+3\big\}, \text{ and}\\
T_j &\defeq S_j + 2i = \big\{z\in \C\colon 5^{j+1} + 7 < 4 \im z < 5^{j+1}+11 \big\}.
\end{align*}

Observe that
  \begin{equation}\label{eqn:stripimage}
  \Phi(S_j) = \big\{ z\in \C\colon 5^{j+2}-5 < 4 \im z < 5^{j+2}+15 \big\} \supseteq S_{j+1} \cup T_{j+1} \end{equation}
for $j\geq 0$. (Compare Figure~\ref{fig:strips}.) Define $S\defeq \bigcup_{j=0}^{\infty} S_j$. 

\begin{lem}\label{lem:strips}
  There is an $\eps >0$ with the following property. Suppose that 
    $f\colon S\to\C$ is analytic and $\lvert f(z) - \Phi(z)\rvert \leq \eps$ for all $z\in S$. 
 
  Then $\lvert \re f(z)\rvert  > 4\lvert \re z\rvert$ for all $z\in S$ with $\lvert \re z\rvert \geq 1$. Furthermore, for every $j\geq 1$, there is a domain $V_j = V_j(f)\subseteq S_0$ such that 
  $f^j\colon V_j \to T_j$ is a conformal isomorphism and
  $f^k(V_j)\subseteq S_k$ for $k=0,\dots,j-1$. Moreover, $\textup{Re}\,z$ is unbounded from above and below on $V_j$, and $2\leq \lvert f'\rvert\leq 8$ 
  on $\bigcup_{k=0}^{j-1} f^k(V_j)$. \end{lem}
\begin{proof}
  For $k\geq 0$, define 
     \begin{align*} \tilde{S}_k &\defeq \big\{z\in S_k\colon \dist(z,\partial S_k) \geq 1/10  \big\}  \\
           &=
             \big\{z\in\C\colon 5^{k+1}- 3/5 \leq   4\im z \leq 5^{k+1}+ 13/5 \big\}  \\ &=
             \Phi^{-1}\big(\big\{z\in\C\colon 5^{k+2} - 3 \leq 4\im z \leq 5^{k+2} + 13\big\}\big). \end{align*}
    Set $U\defeq S$, $V\defeq \Phi(U)$ and $A\defeq \bigcup_{k=0}^{\infty} \tilde{S}_k$. 
       Then the hypotheses of
    Lemma~\ref{lem:approx1} are satisfied. Let $\eps$ be as given in that lemma, chosen additionally such that
    $\eps<1/2$, and let $f$ be $\eps$-close to $\Phi$ on $U$. Then 
  \[ \lvert \re f(z)\rvert \geq \lvert \re \Phi(z)\rvert - \frac{1}{2} = 5\lvert \re z\rvert - \frac{1}{2} > 4\lvert \re z\rvert \] 
    for all $z\in S$ with $\re z \geq 1$.  
    
    Moreover, $f$ is injective on each $\tilde{S}_k$ 
      with $f(\tilde{S}_k)\supseteq S_{k+1}\cup T_{k+1}$. So we may define
      inductively $V_j^j = T_j$ and, for $k=0,\dots,j-1$, 
         \[ V_j^{k} \defeq (f\vert_{\tilde{S}_{k}})^{-1}(V_j^{k+1}). \]
      Then $V_j \defeq V_j^0$ has the desired properties.       
    Finally, $\Phi'(z)=5$ for all $z\in \C$, which implies the final claim by 
       Lemma~\ref{lem:approx1} (observe that $\eta = \inf_{z\in A}\lvert \Phi'(z)\rvert=5$).
\end{proof}

Define 
   \[ D \defeq \D\Bigl(0,\frac{1}{2}\Bigr) \]
  and observe that $\overline{D}\cap \overline{S}=\emptyset$. 
  The disc $D$ will play the role of the disc $D_{-1}$ from
Section~\ref{sec:uniform}. We define $A_0 \defeq \overline{D}\cup \overline{S}$ and 
     \[ g_0\colon A_0 \to\C; \quad z\mapsto \begin{cases}
                            \Phi(z), &\text{if $z\in\overline{S},$} \\
                            0, &\text{if $z\in\overline{D}.$}\end{cases} \]
  Set
     \[ \eps_0 \defeq \min\Big\{\frac{\eps}{2}, \frac{1}{5}\Big\}, \] 
  where $\eps$ is the number from Lemma~\ref{lem:strips}. By the discussion in Section~\ref{sec:preliminaries}, 
    $A_0$ satisfies the hypotheses of Arakelyan's theorem.  Hence there exists an entire function $f_0$ such that 
     \begin{equation} 
         \lvert f_0(z) - g_0(z)\rvert \leq \eps_0\quad \textup{for } z\in A_0. \end{equation}
  The set $A_0$ and the function $f_0$ will remain fixed with these properties throughout the following sections. 
    Any entire function $f$ that is $\eps_0$-close to $f_0$ (and hence $\eps$-close to $g_0$) 
    on $A_0$  satisfies the hypothesis 
    and hence the conclusion of Lemma~\ref{lem:strips}, and also 
   \begin{equation}\label{eqn:Dinvariant} f(\overline{D})\subseteq D. \end{equation}

\section{Entire functions with wandering compacta}\label{sec:maverickconstruction}
 We now prove Theorem~\ref{thm:main}.     Let $K$, $Z_{\I}$ and $Z_{\BU}$ be as in 
    Theorem~\ref{thm:main}. Without loss of generality, suppose that $K\subseteq T_0$, where $T_0$ is the strip from the previous section.
      We may also assume that $Z_{\BU}\neq \emptyset$, since 
      the case $Z_{\BU} = \emptyset$ is covered already by 
      Theorem~\ref{thm:uniform}. (Alternatively, replace $K$ by $K\cup \{\zeta\}$, where $\zeta\notin K$, and set $Z_{\BU}\defeq \{\zeta\}$.)
      We further assume that no component of $\interior(K)$ contains more than one point of $Z\defeq Z_{\I}\cup Z_{\BU}$. Indeed, any Fatou component
      that intersects $\I(f)$ is contained in $I(f)$, and likewise for $\BU(f)$, so any additional points in the same component of $\interior(K)$ can be omitted without affecting
      the conclusion of the theorem.
      Let $(Z_j)_{j=0}^{\infty}$ be a sequence of finite subsets of $Z$ with
        $Z_{j}\subseteq Z_{j+1}$ and 
      such that $Z = \bigcup_{j=0}^\infty Z_j$. 
      
 \begin{rmk}
   We may always
      augment $Z$ to be infinite. However, it turns out that the proof is
      slightly simpler in the case where $\# Z = 2$, which is the case required for the proof of
      Theorem~\ref{thm:rippon}. So we allow $Z$ to be either finite or
      infinite, and discuss those instances where the assumption $\# Z  =2$ leads to
      simplifications. 
 \end{rmk}
      Let $(\zeta_j)_{j=0}^{\infty}$ be a sequence in $Z_{\BU}$ such that 
        for every $\zeta\in Z_{\BU}$, there are infinitely many $j\geq 0$ such that
        $\zeta_j = \zeta$. 
    
Again choose a sequence $(K_m)_{m=0}^{\infty}$ of compact and full subsets 
      as in Lemma~\ref{lem:Kn}, as well as non-separating and $2^{-m}$-dense subsets
       $P_m\subseteq \partial K_m$. We may choose these sets such that $K_0\subseteq T_0$. 
       Our goal is to prove the following version of Proposition~\ref{prop:main}. 
       Here and in the following, we retain all notation from Section~\ref{sec:scaffolding}, including the disc $D$, the functions $g_0$ and $f_0$, the set $A_0$ 
        and the number $\eps_0$. 
       
       \begin{prop}\label{prop:mainmaverick}
   Let $(N_j)_{j=0}^\infty$ be defined as in~\eqref{eqn:Nj}: $N_0=0$ and $N_{j}=N_{j-1}+(j+1)$ for $j\geq 1$. 
  There is a transcendental entire function~$f$ and an increasing 
   sequence $(m_j)_{j=0}^{\infty}$
     with the following properties:
  \begin{enumerate}[(a)]
    \item $\lvert f(z) - g_0(z)\rvert \leq 2\eps_0$ for $z\in A_0$; in particular, $f(\overline{D})\subseteq D$ and the domains
       $V_{j}(f)$ from Lemma~\ref{lem:strips} are defined for all $j\geq 1$;\label{item:Dinvariant}
    \item $f^{N_j+1}(P_{m_{j}})\subseteq D$ for all $j\geq 0$; \label{item:Pmimage}
    \item $f^{N_j+1}$ is injective on $K_{m_{j+1}}$ for all $j\geq 0$, with
          $f^{N_j+1}(K_{m_{j+1}})\subseteq V_{j+1}(f)$;\label{item:injectiveKm}
    \item $\lvert \re f^{N_j+1}(\zeta_{j})\rvert \leq 1$ for all $j\geq 0$;\label{item:zetajimage}
    \item if $N_j+1\leq n \leq N_{j+1}$ and $\zeta \in Z_j\setminus \{\zeta_j\}$, 
         $\lvert \re f^{n}(\zeta)\rvert \geq j$.\label{item:prescribedbehaviour} 
  \end{enumerate}
\end{prop}

Most of the remainder of the section is dedicated to the proof of Proposition~\ref{prop:mainmaverick}. However, we first show how this
   proposition implies Theorems~\ref{thm:rippon} and~\ref{thm:main}.

\begin{proof}[Proof of Theorem~\ref{thm:main}, using Proposition~\ref{prop:mainmaverick}]
   Similarly as in the proof of Theorem~\ref{thm:uniform}, all points in $P_{m_j}$ 
    eventually map to $\overline{D}$ and remain there by~\ref{item:Dinvariant} and~\ref{item:Pmimage},
    while by~\ref{item:injectiveKm}, 
    \[ f^{N_{j+1}}(K) = f^{N_j+j+2}(K)\subseteq f^{j+1}(V_{j+1}(f)) = T_{j+1} \]
   for all $j\geq 0$, and hence all points in $K$ have unbounded orbits. So $\partial K \subseteq \J(f)$,
   proving~\ref{item:boundaryinJ}.
   
   Furthermore, if $N_j +1 \leq n < N_{j+1}$, then 
    \[ f^n(K)\subseteq f^{n-N_j-1}(V_{j+1}(f)) \subseteq S_{n-N_j-1}, \]
 and, in particular, $f^n(K)\cap T_{j+1}=\emptyset$ for $n<N_{j+1}$. Hence, $K$ is a wandering compactum, 
 and~\ref{item:wanderingcompactum} holds. Together with~\ref{item:boundaryinJ}, this also shows that
  every connected component of $\interior(K)$ is a wandering domain. 

 It remains to prove~\ref{item:maverick}. First, let $\zeta \in Z_{\BU}$. 
   As noted above, all points in $K$ have unbounded orbits. On the other hand, there are infinitely
   many $j\geq 0$ such that $\zeta_j = \zeta$, and hence 
   $\lvert\re f^{N_j+1}(\zeta)\rvert \leq 1$ by~\ref{item:zetajimage}. Furthermore, by~\ref{item:injectiveKm}, 
      $f^{N_j+1}(\zeta)\in V_{j+1}\subseteq S_0$ and hence $1 \leq \im f^{N_j+1}(\zeta) \leq 2$. 
    So $\liminf_{n\to\infty} \lvert f^{n}(\zeta) \rvert \leq 3$, and $\zeta\in \BU(f)$ as required.

On the other hand, let $\zeta\in Z_{\I}$. Let $j_0\geq 0$ be large enough that $\zeta\in Z_j$
    for $j\geq j_0$. If $n\geq N_{j_0}+1$, let $j\geq j_0$ be maximal
    with $n \geq N_j+1$. Then 
   $\lvert\re f^{n}(\zeta)\rvert \geq j$ by~\ref{item:prescribedbehaviour}. Hence $\lvert\re f^{n}(\zeta)\rvert\to+\infty$ as $n\to\infty$, and $\zeta\in \I(f)$.
\end{proof} 
\begin{rmk}\label{rmk:BUmaverick}
  Let $\zeta \in Z_{\BU}$, and $\zeta' \in Z\setminus \{\zeta\}$.  Then~\ref{item:zetajimage} and~\ref{item:prescribedbehaviour} show that
     \begin{equation}\label{eqn:BUmaverick}
        \limsup_{j\to\infty} \dist^{\#}(f^{N_j+1}(\zeta) , f^{N_j+1}(\zeta')) > 0. 
     \end{equation}
     
Recall that, at the beginning of the section, we simplified the set $Z=Z_{\BU}\cup Z_{\I}$ by assuming that 
   no interior component of $K$ contains more than one point
    of $Z$. If two points $z$ and $\tilde{z}$ belong to the same wandering
    Fatou component, then~$\dist^{\#}(f^n(z),f^n(\tilde{z}))\to 0$ as $n\to\infty$. 
    Hence, as noted following Definition~\ref{dfn:maverick-points},~\eqref{eqn:BUmaverick} 
    holds for the function in Theorem~\ref{thm:main} 
    as long as $\zeta$ and $\zeta'$ do not belong to the same connected component of $\interior(K)$.
\end{rmk}

\begin{proof}[Proof of Theorem~\ref{thm:rippon}, using Theorem~\ref{thm:main}]
  Let $X$ be a Lakes of Wada continuum with complementary components $U_0, U_1, U_2$ whose boundary agrees with $X$.
    We may assume that $U_0$ is the unbounded connected component of $\C\setminus X$. 
Let $K = \C\setminus U_0 = \Fill(X)$. Choose
    $\zeta_1\in U_1$ and $\zeta_2\in U_2$, and set 
   $Z_{\I} = \{\zeta_1\}$ and $Z_{\BU} = \{\zeta_2\}$. Now apply Theorem~\ref{thm:main}. 

  Then $U_1$ is an escaping wandering domain, and $U_2$ is an oscillating wandering domain. 
    By \cite[Theorems~1.1 and~1.2(a)]{rippon-stallard11} 
      (or our stronger Theorem~\ref{thm:maverick}), 
    \[ \I(f)\cap \partial U_1 = \I(f)\cap X \subseteq X \] 
   has full harmonic measure when viewed from $U_1$, while 
    \[ \partial U_2 \setminus I(f) = X \setminus I(f) \subseteq X\]
   has full harmonic measure when viewed from
   $U_2$. In particular, these sets are non-empty.
\end{proof}

   We will require the following preliminary observation
     concerning the conformal geometry of the sets $K_m$.        
      For $\zeta\in Z$ and $m\geq 0$, 
   let $U^{\zeta}_m$ denote the connected component of $\interior(K_m)$ containing $\zeta$. 
       
    \begin{lem} \label{lem:hyperbolicdistance}
       Suppose that $\omega\neq \zeta$ and $\zeta,\omega$ belong to the same connected component of $K$, but $\zeta,\omega$~do not belong to the same connected component of $\textup{int}(K)$.
        Then 
         the hyper\-bolic distance 
           $\dist_{U_m^{\zeta}}(\zeta,\omega)$ tends to $\infty$ as $m\to\infty$. 
    \end{lem}
    \begin{proof}
       Let $m\geq 0$ and set $\delta_m \defeq \max_{z\in \partial K} \dist(z,\partial K_m)$. (Note that $\delta_m\to 0$ as $m\to\infty$.) Suppose that $\gamma$ is a curve connecting $\zeta$ and $\omega$ in $U_m^\zeta=U_m^\omega$. Such a curve must contain a point $\alpha\in \partial K$.
        (This is trivial if one of the points belongs to $\partial K$, and otherwise it follows from the fact
        that $\partial K$ separates $\zeta$ and~$\omega$.) At least one of $\zeta$ or $\omega$ has distance at least $|\zeta-\omega|/2$ from~$\alpha$. Without loss of generality, we can assume that this point is $\omega$. Parametrise the subcurve $\tilde{\gamma}\subseteq \gamma$ from $\alpha$ to $\omega$ by arc-length as $\tilde{\gamma}:[0,T]\to U_m^\zeta$. Then
\[       
 T\geq |\alpha-\omega|\geq \frac{|\zeta-\omega|}{2}       \quad      
\textup{and}
\quad
\dist(\tilde{\gamma}(t),\partial U_m^\zeta)\leq \dist(\alpha,\partial U_m^{\zeta}) + t \leq \delta_m+t
\]
        for $t\in [0,T]$. By the standard estimate on the hyperbolic metric
        in a simply connected domain, the hyperbolic density of $U_m^\zeta$ at a point $z$ is bounded below
        by $1/(2 \dist(z,\partial U_m^\zeta))$. It follows that
\begin{align*}        
  \ell_{U_m^\zeta}(\gamma)\geq \ell_{U_m^\zeta}(\tilde{\gamma}) \geq \int_0^T \frac{dt}{2(\delta_m+t)}=\frac{1}{2}\ln \Big(1+\frac{T}{\delta_m}\Big)\geq \frac{1}{2} \ln\Big(1+\frac{|\zeta-\omega|}{2\delta_m}\Big),
   \end{align*}
        and hence
           \[ \dist_{U_m^\zeta}(\zeta,\omega) \geq  \frac{1}{2}\ln\Bigl(1+ \frac{\lvert \zeta-\omega\rvert}{2\delta_m}\Big) \to +\infty\quad \textup{as } m\to\infty. \qedhere \]
    \end{proof}

Recall that
       $U_m^\zeta$ is a Jordan domain; let $\pi_m^{\zeta}\colon U_m^{\zeta}\to\DD$ be a conformal map
       with $\pi_m^{\zeta}(\zeta) = 0$. 
       Let $\omega\in Z\setminus\{\zeta\}$ belong to the same
       connected component of $K$ as $\zeta$. By the assumption on $Z$ made at the beginning of this section, the points
       $\zeta$ and $\omega$ do not both belong to the same interior
       component of $K$. Hence, by Lemma~\ref{lem:hyperbolicdistance},
         $\lvert \pi_m^{\zeta}(\omega)\rvert\to 1$ as $m\to\infty$.

         Passing to a subsequence of $(K_m)_{m=0}^\infty$, we may assume that for every 
           pair of distinct points
         $\zeta,\omega\in Z$, 
         there is a point $\xi^{\zeta}(\omega)\in\partial\DD$ such that 
         $\pi_m^{\zeta}(\omega) \to \xi^{\zeta}(\omega)$ as $m\to\infty$. 

\begin{rmk}
  Suppose that $Z$ consists of exactly two points $\zeta$ and $\omega$ that both belong to the same connected component of $K$. Then we may normalise
    $\pi_m^{\zeta}$ and $\pi_m^{\omega}$ so that they map the other point to the positive
    real axis, and the above property is automatic with $\xi^{\zeta}(\omega) = \xi^{\omega}(\zeta) = 1$. 
\end{rmk}

  Recall from Proposition~\ref{prop:mainmaverick} that, at time $N_j+1$, the point
   $\zeta_{j}$ should map to a bounded part of the set $V_{j+1}(f)$, while other points of $Z_j$
   should map sufficiently far to the right. In order to achieve this, we shall use the following fact about 
  conformal mappings.  

\begin{figure}
\begin{center}
\def\svgwidth{\textwidth}
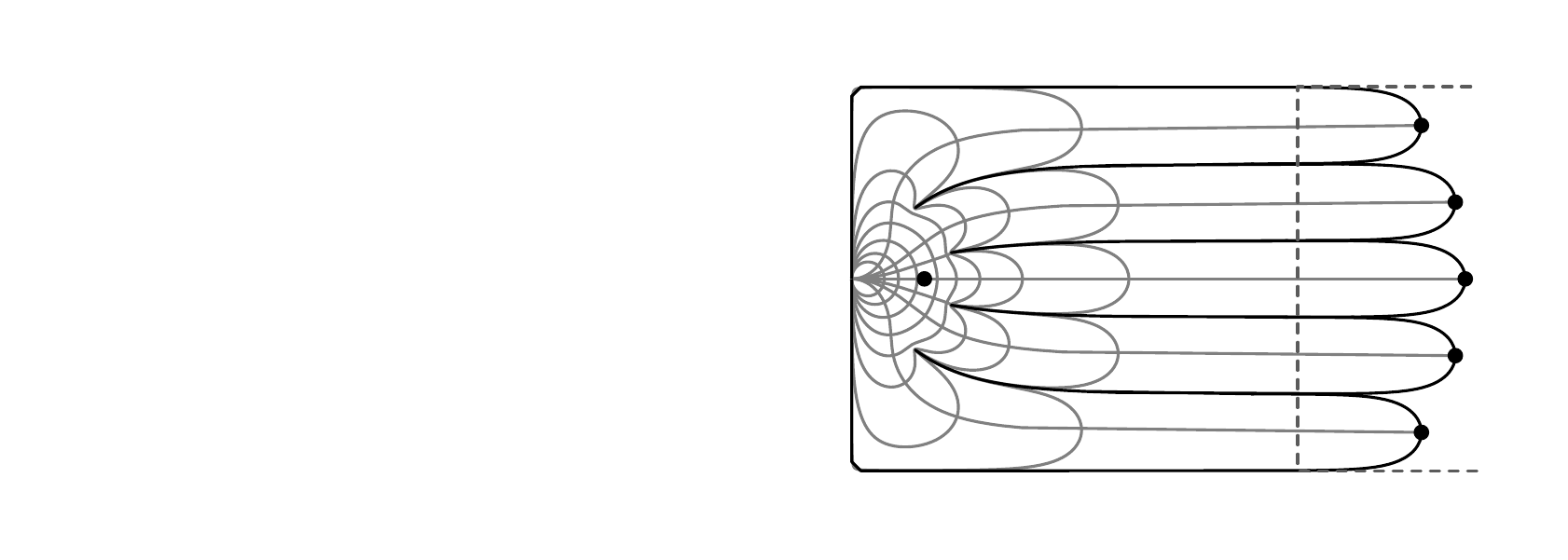 
\end{center}
\vspace*{-20pt}
\caption{Illustration of Proposition~\ref{prop:spikes}. Here 
 $V$ is a half-strip, $\nu\in V$ and $\Delta\subseteq V$ is a further half-strip, far to the 
 right of $\nu$. The set $\Xi\subseteq\partial\DD$ consists of  five points, which are mapped into $\Delta$ by $\phi\colon \DD\to W\subseteq V$.}\label{fig:confspikes}
\end{figure}

\begin{prop}[Conformal maps with prescribed behaviour]\label{prop:spikes}
   Let $V\subseteq\C$ be a domain, let $\Delta \subseteq V$ be non-empty and open, and let $\nu\in V$.  
     Let $\Xi\subseteq \partial\DD$ be a compact set of zero logarithmic capacity.  
       
   Then there exist a bounded Jordan domain $W\subseteq V$ 
     and a conformal isomorphism 
       \[ \phi \colon \DD\to W \]
      with $\phi(0)=\nu$ whose continuous extension to $\overline{\DD}$ satisfies 
     $\phi(\xi) \in \Delta$ for all $\xi \in \Xi$. 
 \end{prop}
 
 Proposition~\ref{prop:spikes} follows from a result of Bishop~\cite{bishop2006} concerning interpolation using boundary values of
   conformal maps. It will be applied in settings where $\nu$ and $\Delta$ are very far apart in $V$ (see Figure~\ref{fig:confspikes}).
     To avoid interrupting the flow of ideas, we postpone the proof of Proposition~\ref{prop:spikes} until Section~\ref{sec:spikes}, but remark that most of the
      applications in this paper do not 
   use the full strength of the result. For the proof of Theorem~\ref{thm:main}, 
     we use only the case where $\# \Xi < \infty$. Moreover, the case when $\# \Xi = 1$ is trivial: 
     choose $W\subseteq V$ to be a Jordan domain containing $\nu$ and whose boundary passes through $\Delta$, and let $\phi$ be an
     appropriately normalised conformal isomorphism. This trivial case corresponds to the setting where $\# Z = 2$ which, as noted above, is 
     sufficient to prove 
     Theorem~\ref{thm:rippon}. We also remark that Proposition~\ref{prop:spikes} is not used in the proof of Theorem~\ref{thm:eremenkocounterexample}.

\begin{proof}[Proof of Proposition~\ref{prop:mainmaverick}]
The construction follows a similar pattern as the proof of Proposition~\ref{prop:main}. We
   inductively construct a sequence $(f_j)_{j=0}^{\infty}$ of entire functions, where $f_j$ approximates
   a function $g_j$~-- continuous on a closed (but unbounded)
    set $A_j$ and holomorphic on its interior~-- up to a uniform error of at most $\eps_j$. (Recall that $f_0$, $g_0$, $A_0$ and $\eps_0$ were already chosen in
    Section~\ref{sec:scaffolding}.) For $j\geq 0$, define
   \[ \Sigma_j \defeq \overline{S} \cup \bigl\{ z\in\C\colon \lvert 4\im z \rvert \leq 5^{j+1} + 3\bigr\}. \]
  Then $A_0\subseteq \Sigma_j$, $\Sigma_j \cap T_j = \emptyset$, $T_{\ell}\subseteq \Sigma_j$ for $\ell < j$ and
  $\bigcup_{j=0}^\infty \Sigma_j = \C$. 

\begin{figure}
\begin{center}
\def\svgwidth{\textwidth}
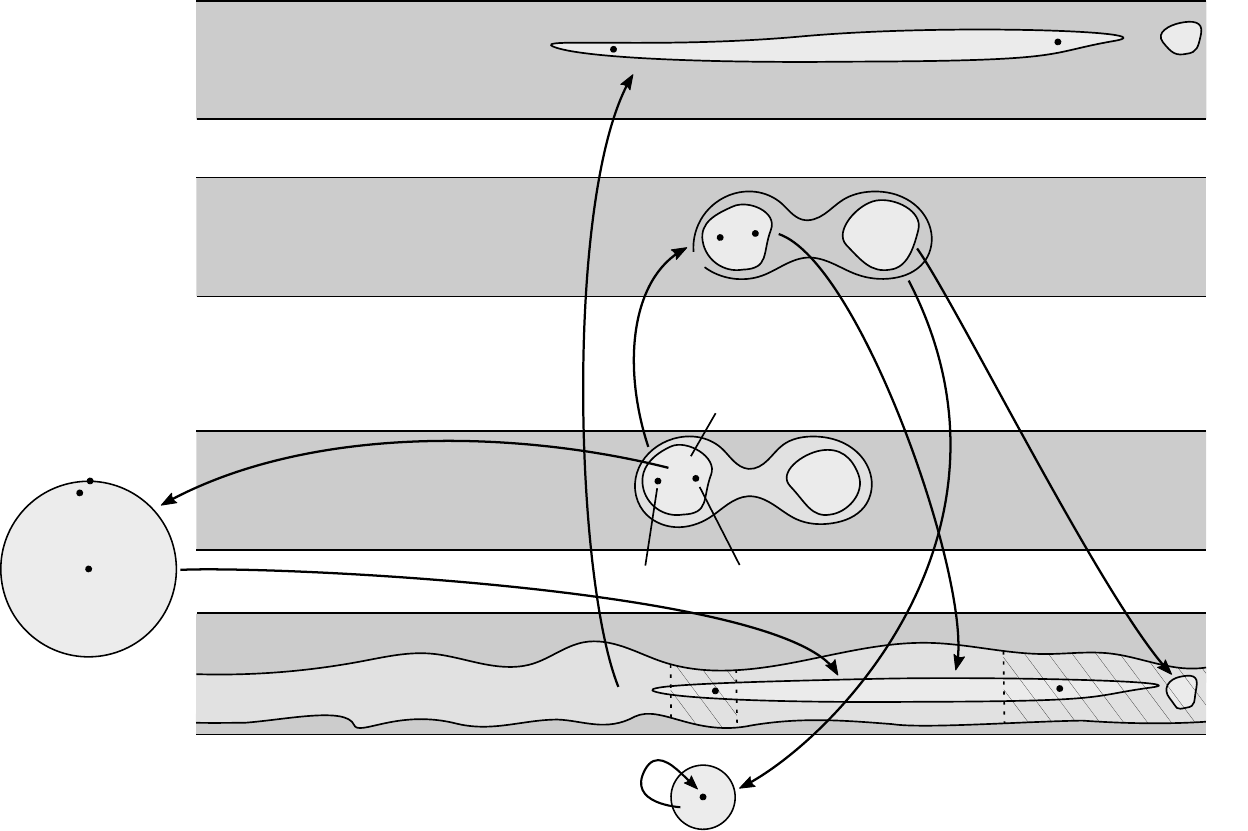 
\end{center}
\caption{The definition of $g_{j+1}$ in the proof of Proposition~\ref{prop:mainmaverick}.   The set $L_j\subseteq K_{m_j}\subseteq T_0$ (shown in light grey) has two connected components, one of which is $\overline{U_j}$, containing $\zeta_j$. On
  $f_j^{N_j}(U_j)$, the map $g_{j+1}$ is the composition of three conformal isomorphisms: $\bigl(f_j^{N_j}|_{U_j}\bigr)^{-1}$, 
    the Riemann map $\pi_j\colon U_j \to \DD$ and $\phi_j\colon \DD\to W_j\subseteq V_{j+1}$, as obtained from Proposition~\ref{prop:spikes}.}
    \label{fig:mainmaverickproof}
\end{figure}

The sequence $(m_j)_{j=0}^{\infty}$ is also chosen as part of the inductive 
   construction, which ensures the following inductive hypotheses. 
\begin{enumerate}[(i)]
 \item For every $j\geq 0$, $f_j^{N_j}$ is injective on $K_{m_j}$, and $f_j^{N_j}(K_{m_j})\subseteq T_j$.
   \label{item:injectivity_again} 
 \item For every $j\geq 1$, $\Sigma_{j-1}\subseteq A_j \subseteq \Sigma_{j}$.\label{item:nested_again}
 \item $\eps_j \leq \eps_{j-1}/2$ for $j\geq 1$. \label{item:inductiveclose_again}
\end{enumerate} 
Set $m_0\defeq 0$; the inductive
 hypothesis~\ref{item:injectivity_again} holds trivially for $j=0$. 

 Suppose that $f_{j}$ and $m_{j}$ have been constructed. Let 
   $R_j\subseteq T_j$ be a compact full neighbourhood of $Q_j \defeq f_{j}^{N_{j}}(P_{m_{j}})$ disjoint 
    from $f_{j}^{N_{j}}(K_{m_{j}+1})$. 

 Let $Z'_{j}$ consist of those 
    $\omega\in Z_{j}\setminus\{\zeta_{j}\}$ that belong to the same connected component of $K$ as $\zeta_j$. 
    Define 
     \[ \Xi_j \defeq \big\{ \xi^{\zeta_j}(\omega)\colon \omega\in Z_{j}' \big\} \subseteq \partial\DD. \]

  The inductive hypothesis \ref{item:inductiveclose_again} implies, in particular, that $f_{j}$ is $2\eps_0$-close to $g_0$ on $A_0$. Therefore
   the domains $V_{j+1} = V_{j+1}(f_{j})$ from Lemma~\ref{lem:strips} are defined. 
   Define
     $\Delta_j \defeq \{z\in V_{j+1}\colon \re z > j\}$ and $\tilde{\Delta}_j\defeq \{z\in V_{j+1}\colon \lvert \re z\rvert<1\}$.
     Choose any $\nu_j\in \tilde{\Delta}_j$ and apply Proposition~\ref{prop:spikes} with $V=V_{j+1}$, $\Delta=\Delta_j$, $\nu=\nu_j$
      and $\Xi=\Xi_j$. We obtain a Jordan domain $W_j\subseteq V_{j+1}$ and 
    a conformal map $\phi_j\colon \DD\to W_j$, extending continuously to $\partial\DD$, such that 
    $\phi_j(0)  = \nu_j \in \tilde{\Delta}_{j}$ and 
    $\phi_j(\xi) \in \Delta_{j}$ for all $\xi\in\Xi_j$.

   Now choose $m_{j+1}\geq m_{j} + 2$ sufficiently large that 
    \[ \phi_j( \pi_{m_{j+1}-1}^{\zeta_j}(\omega)) \in \Delta_j \] 
    for all $\omega \in Z_{j}'$, and such that 
      $\omega\notin U_{m_{j+1}-1}^{\zeta_j}$ for $\omega\in Z_{j}\setminus (Z_{j}'\cup\{\zeta_j\})$. 
      (Recall that $\pi_{m}^{\zeta_j}(\omega) \to 
        \xi^{\zeta_j}(\omega) \in \Xi_j$ as $m\to\infty$ for $\omega\in Z_{j}'$.) We set 
        $L_j \defeq K_{m_{j+1}-1}$, $U_j\defeq U_{m_{j+1}-1}^{\zeta_j}$ and $\pi_j\defeq \pi_{m_{j+1}-1}^{\zeta_j}$.
        Recall that $L_j$ is a finite disjoint union of closed Jordan domains, one of which is $\overline{U_j}$. In particular,
        the conformal map $\pi_j\colon U_j\to\DD$ extends continuously to $\overline{U_j}$.

  Let $\alpha_{j}$ be an affine map such that 
     \[\alpha_j(f_{j}^{N_{j}}(L_j)) \subseteq \{z\in V_{j+1} \setminus \overline{W_j}\colon \re z > j\}= \Delta_j\setminus \overline{W_j}. \]
   Set
     \[ A_{j+1}\defeq \Sigma_j \cup R_j \cup f_{j}^{N_{j}}(L_j) \subseteq \Sigma_j \cup T_j \subseteq \Sigma_{j+1}. \]
   The set $\Sigma_j$ is the disjoint union of a large closed central horizontal strip and the
      strips $\overline{S_k}$ for $k > j$. It is disjoint from $T_j$, and hence from
         $R_j$ and $f_j^{N_j}(L_j)$. The latter two sets are compact and full, and disjoint from each other by choice of $R_j$. 
      Hence $A_{j+1}$ satisfies the hypotheses of Arakelyan's theorem. 
        We define
   \[ g_{j+1} \colon A_{j+1}\to\C; \quad z\mapsto\begin{cases}
            f_{j}(z), & \text{if $z\in \Sigma_j$}, \\
             0, &\text{if $z\in R_j$}, \\
            \phi_j(\pi_j( (f_{j}^{N_{j}}|_{L_j})^{-1}(z))), &\text{if
             $z\in f_j^{N_j}(\overline{U_j})$}, \\
             \alpha_{j}(z), &\text{if $z\in f_j^{N_j}(L_j\setminus\overline{U_j})$}. \end{cases} \]
  (See Figure~\ref{fig:mainmaverickproof}.) Then $g_{j+1}$ is continuous on $A_{j+1}$ and holomorphic on its interior. 

   Finally, choose $\eps_{j+1}$ according to~\ref{item:inductiveclose_again} and sufficiently small that any entire 
     function $f$ with $\lvert f(z) - g_{j+1}(z)\rvert \leq 2\eps_{j+1}$ on $A_{j+1}$ satisfies: 
 \begin{enumerate}[(1)]
   \item $f^{N_{j}+1}(P_{m_j})\subseteq D$; \label{item:Pmimage_induction}
   \item $f^{N_{j+1}}$ is injective on $K_{m_{j+1}}\subseteq \interior(L_j)$;\label{item:injectiveLj}
   \item $f^{N_j+\ell}(K_{m_{j+1}})\subseteq S_{\ell-1}$ for $\ell=1,\dots,j+1$;\label{item:mappingthroughstrips}
   \item $f^{N_{j+1}}(K_{m_{j+1}}) \subseteq T_{j+1}$; \label{item:imageinT}
   \item $f^{N_{j}+1}(\zeta_j) \in \tilde{\Delta}_j$;\label{item:zetamapsback}
   \item $f^{N_j+1}(\omega) \in \Delta_j$ for $\omega\in Z_j\setminus \{\zeta_j\}$.\label{item:omegamapsright}
 \end{enumerate}
  Note that $g_{j+1}$ itself satisfies these properties,  and they are preserved under sufficiently 
   close approximation by Lemma~\ref{lem:approx2} and Corollary~\ref{cor:approx}. We complete the inductive construction by applying
   Arakelyan's theorem to find an entire function 
   $f_{j+1}$ that is $\eps_{j+1}$-close to $g_{j+1}$ on $A_{j+1}$. The inductive hypothesis~\ref{item:injectivity_again} follows from~\ref{item:injectiveLj} and~\ref{item:imageinT},
   while~\ref{item:nested_again} and~\ref{item:inductiveclose_again} hold by our choice of $A_{j+1}$ and $\eps_{j+1}$.

   By~\ref{item:nested_again} and~\ref{item:inductiveclose_again},
     \begin{equation}\label{eqn:fandgclose} \lvert f_{k}(z) - g_j(z)\rvert \leq \sum_{\ell=j}^{k} \eps_{\ell} < 2\eps_j \end{equation}
      whenever $k\geq j\geq 0$ and $z\in A_j\supseteq \Sigma_{j-1}$. It follows that $f_k$ converges locally uniformly to a limit function $f\colon\C\to\C$,
      which satisfies claim~\ref{item:Dinvariant} of Proposition~\ref{prop:mainmaverick} by~\eqref{eqn:fandgclose}.
   Claim~\ref{item:Pmimage} holds by~\ref{item:Pmimage_induction}. 
   Claim~\ref{item:injectiveKm} holds by~\ref{item:injectiveLj},~\ref{item:mappingthroughstrips} and~\ref{item:imageinT}.
   Claim~\ref{item:zetajimage} follows from~\ref{item:zetamapsback}.
   Finally,~\ref{item:prescribedbehaviour} follows from~\ref{item:omegamapsright} together with Lemma~\ref{lem:strips}. 
\end{proof}
\begin{rmk}\label{rmk:omegalimit}
  With a slight modification of the proof, we may achieve that all points $\zeta\in Z_{\BU}$ have the same $\omega$-limit set. 
Indeed, let $f$ be any map as in Lemma~\ref{lem:strips}, 
and consider the curves
       \[ \gamma_f^{j}\colon \R\to V_j; \qquad t\mapsto (f^j|_{V_j})^{-1}\Bigl(5^j\cdot t + \frac{5^{j+1}+9}{4}i \Bigr).  \] 
       Since $f$ and $\Phi$ are close to each other and 
      both are uniformly expanding on $S$, it is straightforward to see that $\gamma_f^j$ converges uniformly to a curve
        $\gamma_f\colon \R\to S_0$, which consists precisely of the points $z$ with $f^j(z)\in S_j$ for all $j\geq 0$. 
     Moreover, again because of the expansion of $f$, the distance $\lvert \gamma_f^j(t) - \gamma_f(t)\rvert$ is bounded by
     a constant that depends only on $j$ (not on $f$ or $t$), and tends to zero as $j\to\infty$.

 In the proof of Proposition~\ref{prop:mainmaverick}, when defining $g_{j+1}$, 
    we chose $\nu_j$ to have small real part, but could have chosen it to be any point of 
    $V_{j+1}$, as long as for every
     $\zeta \in Z_{\BU}$, there is an infinite set of $j$ with $\zeta = \zeta_j$ such that $\re \nu_j$ is bounded independently of
     $j$. Let
     $(t_j)_{j=0}^{\infty}$ be a sequence of rational numbers so that, for every $t\in \Q$ and every $\zeta\in Z_{\BU}$, 
     there are infinitely many $j$ such that 
     $t_j = t$ and $\zeta_j = \zeta$. In the proof of Proposition~\ref{prop:mainmaverick}, we then adjust the definition of $\nu_j$ and $\tilde{\Delta}_j$, taking
     $\nu_j \defeq \gamma_{f_j}^{j+1}(t_j)$ and $\tilde{\Delta}_j \defeq D(\nu_j,1/j)\cap V_{j+1}$ for all $j$. If $\eps_{j+1}$ is chosen sufficiently small, and
       $f$ is any function with $\lvert f(z) - g_{j+1}(z)\rvert \leq 2\eps_{j+1}$ on $A_{j+1}$, then 
       $f^{N_j+1}(\zeta_j)$ and  $\gamma_{f}^{j+1}(t_j)$ both belong to $\tilde{\Delta}_j$. (This replaces condition~\ref{item:zetamapsback} in the choice of $\eps_{j+1}$.)
       
For the resulting limit function $f$, the $\omega$-limit set of any 
      $z\in K$ is contained in the union $\Omega$ of $\gamma_f$, its forward iterates $f^n(\gamma_f)\subseteq S_n$, and $\infty$.  
        On the other hand, let $\zeta\in Z_{\BU}$. By construction, the $\omega$-limit set of $\zeta$ contains 
        $\gamma_f(t)$ for all $t\in\Q$, and hence agrees with $\Omega$.
                
  Observe that, if $Z_{\BU}$ contains a point of $\interior(K)$, then the resulting function $f$ has a wandering domain
    on which the set of (constant) limit functions of the iterates contains unbounded 
    closed connected sets. Lazebnik~\cite{lazebnik17} previously showed that 
    the set of limit functions in a wandering domain can be uncountable, answering
    a question raised by
    Osborne and Sixsmith~\cite[Question~2]{osborne-sixsmith16}.
\end{rmk}

\section{Larger sets of  maverick points}

\label{sec:mavericklarge}

In this section, we indicate how we may modify the proof of Proposition~\ref{prop:mainmaverick}
    to obtain examples with larger sets of maverick points. In particular, we prove 
    Theorems~\ref{thm:capacity} and~\ref{thm:HD}. Recall that Theorem~\ref{thm:capacity} says that in the case that $U=\mathbb{D}$, we may obtain any compact subset of $\partial \mathbb{D}$ of logarithmic capacity zero as the set of maverick points of $U$. On the other hand, Theorem~\ref{thm:HD} shows that for some domains, the set of maverick points may have positive planar Lebesgue measure. 

 As already mentioned, in the proof of Proposition~\ref{prop:mainmaverick}, we do not use the full strength of Proposition~\ref{prop:spikes}, which we apply only to
   a finite set $\Xi_j$. Recall that, in Proposition~\ref{prop:spikes}, the compact set $\Xi$ may be chosen to be infinite, as long as it has 
   zero logarithmic capacity. Another way in which we may modify the proof of Proposition~\ref{prop:mainmaverick} is that we might choose to reverse the roles
   of $\Delta$ and $\tilde{\Delta}$, so that the points of the sequence $\zeta_j$ become escaping points rather than elements of $\BU(f)$. 
   
   To investigate the additional flexibility this gives us in the proof, let us follow the set-up and notation as in Section~\ref{sec:maverickconstruction}, but modify the parts of the construction that were 
    related to the sets 
    $Z_{\I}$ and $Z_{\BU}$. As before,
   we suppose that $K\subseteq T_0$ is a compact, full set. For simplicity, we restrict here to the case 
    where $K$ is the closure of a simply connected domain $U=\interior(K)$, which is the case of interest for Theorems~\ref{thm:capacity} and~\ref{thm:HD}. 
    We also fix a point $\zeta_0\in U$ and $k\in\{0,1\}$. Our goal is to construct a function $f$ for which $U$ is 
    an oscillating (if $k=0$) or escaping (if $k=1$) wandering domain, and for which $\partial U$ contains a set of maverick points on its boundary that is ``large'' in some sense. 
    Theorem~\ref{thm:main} shows that we may realise any countable set $Z_{\mav}\subseteq \partial U$ as a set of maverick points. 
    Indeed, we may set  $Z_{\BU} = \{\zeta_0\}$ and $Z_{\I} = Z_{\mav}$ if $k=0$, or $Z_{\BU} = Z_{\mav}$ and $Z_{\I} = \{\zeta_0\}$ if $k=1$. 
    
  Now, instead of $Z_{\I}$ and $Z_{\BU}$, suppose that we start with a given $\sigma$-compact set $Z_{\mav}\subseteq \partial U$. For $j\geq 0$, also fix 
    \begin{enumerate}[(i)]
      \item non-empty compact sets $Z_j \subseteq Z_{\mav}$ with the property that, for every $z\in Z_{\mav}$, there are infinitely many $j$ such that $z\in Z_j$; 
      \item compact sets $\Xi_j \subseteq \partial\DD$ of zero logarithmic capacity.\label{item:Xij}
    \end{enumerate}
   Once again, we choose a nested sequence $K_m\subseteq T_0$ of closed Jordan domains   shrinking to $K$, and 
     a conformal map $\pi_m\colon U_m \to \mathbb{D}$ with $\pi_m(\zeta_0)=0$, where $U_m = \interior(K_m)$. The key assumption is
     that this choice can be made in such a way that
    \begin{enumerate}[(*)]
      \item for every $j\geq 0$,  we have $\pi_m(Z_j)\to \Xi_j$ in the Hausdorff metric
        as $m\to\infty$.\label{item:convergenceassumption}
    \end{enumerate}
    
    With this setup, we follow the proof of Proposition~\ref{prop:mainmaverick} almost word-for-word, except for the following modifications. 
     \begin{itemize}
        \item We take $\zeta_j \defeq \zeta_0$ for all $j$. 
        \item The role of $Z_j'$ is taken by $Z_j$ (since $\zeta_j\notin Z_j$ and $K$ is connected). 
        \item The set $\Xi_j$ is the set from~\ref{item:Xij}. 
        \item If $k=1$, then the roles of $\Delta_j$ and $\tilde{\Delta}_j$ are exchanged. That is, in this case
                $\tilde{\Delta}_j \defeq \{z\in V_{j+1}\colon \re z > j\}$ and $\Delta_j\defeq \{z\in V_{j+1}\colon \lvert \re z\rvert<1\}$.
        \item The map $\alpha_j$ is not required, since $L_j = \overline{U_j}$. 
    \end{itemize}
    
    \begin{prop}[General construction of maverick points]\label{prop:generalmaverick}
      Using the above notation, and assuming~\ref{item:convergenceassumption}, there exists a transcendental entire function $f$ with the following properties. 
       \begin{enumerate}[(a)]
         \item $U$ is an oscillating wandering domain if $k=0$, and an escaping wandering domain if $k=1$. 
         \item Every point of $Z_{\mav}\subseteq \partial U$ is a maverick point. 
         \item If $k=0$, then every point $z\in Z_{\mav}$ such that $z\in Z_j$ for all but finitely many $j$ belongs to the escaping set.\label{item:escapingmaverick} 
       \end{enumerate}
    \end{prop}
    \begin{proof}
      That $U$ is a wandering domain follows exactly as in the proof of Theorem~\ref{thm:main}. We have 
         $f^{N_j+1}(\zeta_0) \in \tilde{\Delta}_j$ for all $j$. If $k=0$, then $\lvert \re f^{N_j+1}(\zeta_0)\rvert < 1$ for all $j$; hence
          $\zeta_0\in \BU(f)$ and $U$ is an oscillating wandering domain. If $k=1$, then $\re f^{N_j+1}(\zeta_0) > j$ for all $j$, and by Lemma~\ref{lem:strips}, also 
          $\re f^{n}(\zeta_0)>j$ for $N_j+1\leq n \leq N_{j+1}$. Hence $\zeta_0\in I(f)$ and $U$ is an escaping wandering domain. 

    For $j\geq 0$, we have $f^{N_j+1}(\zeta_0)\in \tilde{\Delta}_j$, but $f^{N_j+1}(z)\in \Delta_j$ for all $z\in Z_j$. In particular, 
        $\limsup_{j\to \infty} \dist^{\#}(f^{N_j+1}(z), f^{N_j+1}(\zeta_0)) > 0$ for all $z\in Z_{\mav}$. Hence every such $z$ is a maverick point (see Lemma~\ref{lem:maverick}
        below).

      If $k=0$ and $z\in Z_j$ for all but finitely many $j$, then $\re f^{N_j+1}(z) > j$ for all such $j$, and it follows as in the proof
        of part~\ref{item:prescribedbehaviour} of Proposition~\ref{prop:mainmaverick} that $\re f^{n}(z) > j$ for $N_j+1\leq n \leq N_{j+1}$. Hence 
        $z\in I(f)$.
    \end{proof}
    
We now apply this observation to prove Theorems~\ref{thm:capacity} and~\ref{thm:HD}.
          
\begin{proof}[Proof of Theorem~\ref{thm:capacity}]
  Let $K_m = \overline{D(\zeta_0, r_m)}$ be a nested sequence of concentric closed discs shrinking
   to a disc $K= \overline{D(\zeta_0,r)}$, where $\zeta_0\in T_0$ and $r_0$ is chosen sufficiently small such that $K_0\subseteq T_0$. 
   Let $\pi_m(z) = (z-\zeta_0)/r_m$; then $\pi_m$ maps $U_m \defeq \interior(K_m)$ to $\DD$. 
     Define 
        \[ Z_{\mav} \defeq \{ \zeta_0 + r\xi \colon  \xi\in \Xi\}. \]
    Then $\pi_m(Z_{\mav})= \Xi$ for all $m$, so property~\ref{item:convergenceassumption}, is satisfied.
    
    Let $k\in \{0,1\}$, and apply Proposition~\ref{prop:generalmaverick} with
      $Z_j = Z_{\mav}$  and $\Xi_j = \Xi$ for all $j$. We obtain an entire function for which $U = D(\zeta_0,r)$ is 
       a wandering domain (oscillating or escaping depending on $k$) and all points of $Z_{\mav}$ are maverick points. 
       Conjugating with an affine map that takes $K$ to $\overline{\DD}$, the proof is complete.
\end{proof}

\begin{proof}[Proof of Theorem~\ref{thm:HD}]
 Let $Z\subseteq \C$ be a Jordan arc of positive Lebesgue measure, let 
    $\gamma\colon (0,1)\to \C\setminus Z$ be an injective analytic curve such that as $t\to 0$, $\gamma(t)$ accumulates on
    all of $Z$ from one side, while as $t\to 1$, $\gamma(t)$ accumulates on all of $Z$ from the other side. Then $\C\setminus (Z\cup \gamma)$ has
    exactly two connected components by a generalised version of the Jordan curve theorem; see e.g.~\cite[Theorem~5.7]{timorinmoore}.
   Let $U$ be the bounded connected component of $\C\setminus (Z\cup \gamma)$. Then $U$ is a regular simply connected domain with 
   $\partial U = Z\cup \gamma$; we set $K\defeq \overline{U}$. Let $\zeta_0\in U$, and observe that $Z$ has zero harmonic measure in $U$. (Indeed, by construction 
   all points of $Z$, with the exception of one endpoint, are inaccessible from $U$.)
    
   Choose  a nested sequence of closed Jordan domains $(K_m)$ shrinking down to $K$. Set $U_m = \interior(K_m)$ and let $\pi_m\colon U_m\to \DD$ be a conformal
   isomorphism with $\pi_m(\zeta_0)=0$. Then the Euclidean diameter of $\pi_m(Z)$ tends to zero as $m\to\infty$, so we may normalise $\pi_m$ in such a way that 
   $\pi_m(Z)\to \{1\}$ as $m\to\infty$. 
   
   We assume that the above sets were chosen such that $K_0\subseteq T_0$. For $j\geq 0$, set
    $Z_j \defeq Z_{\mav}\defeq Z$, $\Xi_j \defeq \{1\}$. Let $k\in \{0,1\}$, and apply 
     Proposition~\ref{prop:generalmaverick} with either $k=0$ or $k=1$. 
     We obtain a function $f$ for which $U$ is an escaping (if $k=1$) or oscillating (if $k=0$) wandering domain, and for which the set of
     maverick points contains $Z$ and hence has positive Lebesgue measure. The set of non-maverick points is contained in the 
    analytic curve $\gamma = \partial U\setminus Z$, and hence has Hausdorff dimension $1$. 
\end{proof}

\begin{rmk}
 For the functions constructed in the proofs of Theorems~\ref{thm:capacity} and~\ref{thm:HD}, if the wandering domain is oscillating, then
  all points of $Z_{\mav}$ are in fact escaping by Proposition~\ref{prop:generalmaverick}~\ref{item:escapingmaverick}.
\end{rmk}

\section{Counterexamples to Eremenko's conjecture}\label{sec:eremenko}

We prove Theorem~\ref{thm:eremenkocounterexample} first in the case 
  where $X=\{0\}$, and then indicate how to modify the proof in order to obtain the more general case.

\begin{thm}\label{thm:eremenkopoint}
  There is a transcendental entire function $f$ such that  \begin{enumerate}[(a)]
    \item $[0,+\infty)\subseteq J(f)$;
    \item $0\in \I(f)$;
    \item $(0,+\infty)\subseteq \BU(f)$;
    \item $[0,+\infty)$ is a connected component of $\J(f)\cup \I(f)\cup \BU(f)$.
  \end{enumerate}  
  In particular, $\{0\}$ is a connected component of $I(f)$.  
\end{thm}

In the proofs of Theorems~\ref{thm:eremenkopoint} and~\ref{thm:eremenkocounterexample}, we shall use conformal maps 
  between domains that agree with strips close to infinity. The following simple fact will allow us to control their derivatives.
  
\begin{lem}[Conformal maps between extensions of half-strips]\label{lem:stripderivative}
  Let $U_1,U_2\subseteq\C$ be simply connected domains such that, for sufficiently large $R>0$ and $j\in \{1,2\}$, 
    \[ \{ \im z\colon z\in U_j \text{ and } \re z = R\} = (-\pi/2,\pi/2). \] 
   Let $\phi\colon U_1\to U_2$ be a conformal map such that $\re \phi(z)\to +\infty$ as $\re z\to +\infty$. 
     Then 
        \[ \phi'(z) = 1 + O(\exp(-\re z))  \qquad\text{and}\qquad
           \phi(z) = z + \rho + O(\exp(-\re z)) \]
     as $\re z\to+\infty$, where $\rho\in\R$ is a constant. 
\end{lem}
\begin{proof}
  Restricting $\phi$, if necessary, we may assume that $U_1$ and $U_2$ are contained in the strip of height $\pi$ centred at the real axis. 
    Write $w = \phi(z)$, $\zeta = \exp(-z)$ and $\omega = \exp(-w)$. Since $\exp$ is injective on $U_1$ and $U_2$, setting
    $\psi(\zeta) \defeq \omega$ yields a well-defined
    conformal map $\psi\colon \exp(-U_1)\to \exp(-U_2)$. By the Schwarz reflection principle, $\psi$ extends conformally to a neighbourhood of $0$;
    set $\lambda\defeq \psi'(0)>0$ and $\rho \defeq -\log \lambda$. 
    
  We have $\omega = \lambda\zeta + O(\lvert \zeta\rvert^2)$, $1/\omega = 1/(\lambda \zeta) + O(1)$ and $\psi'(\zeta)=\lambda +O(\lvert \zeta\rvert)$
    as $\zeta\to 0$ (and hence as $\re z\to +\infty$). Thus
    \[ \phi'(z) = -\zeta \cdot \psi'(\zeta) \cdot \frac{-1}{\omega} = 
                         \frac{\lambda \zeta}{\omega} \cdot (1 + O(\lvert \zeta\rvert)) = 1 + O(\lvert \zeta\rvert) = 1 + O(\exp(-\re z)) \]
   as claimed. Similarly,
     \[ \phi(z) = \Log\Bigl(\frac{1}{\omega}\Bigr) = \Log\Bigl(\frac{1}{\lambda \zeta} + O(1)\Bigr) = z + \rho + O(\lvert \zeta\rvert). \qedhere \]
\end{proof}

\begin{figure}
\begin{center}
\def\svgwidth{.65\textwidth}
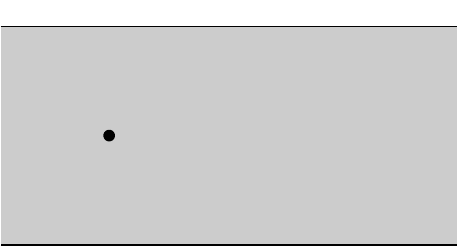 
\end{center}
\caption{The unbounded closed set $K=\bigcap_{j= 0}^\infty K_j$ in the proof of Theorem~\ref{thm:eremenkopoint}.}\label{fig:eremenko}
\end{figure}
\begin{proof}[Proof of Theorem \ref{thm:eremenkopoint}]
  We shall prove the theorem with $[0,\infty)$ replaced by $K\defeq [0,\infty) + 7i/2$; the result then follows by conjugating with a translation. 
    Let the strips $(S_j)$, $(T_j)$, the disc $D$, the map $f_0$ and the numbers $N_j$ and  
    $\eps_0$
    be as in Section~\ref{sec:scaffolding}. Also recall the definition of the domains $V_j(f)$ from Lemma~\ref{lem:strips}. 
    Observe that $K\subseteq T_0$. 
    
    Similarly as in Proposition~\ref{prop:mainmaverick}, we 
     construct an entire function $f$ and a sequence
      $(K_j)_{j=0}^{\infty}$ with the following properties for all $j\geq 0$, where 
        $\zeta \defeq 7i/2$ and $t_j\defeq j+2$ (see Figure~\ref{fig:eremenko}):
     \begin{enumerate}[(a)]
       \item $K_j$ is a closed horizontal half-strip with $K\subseteq K_{j+1}\subseteq \interior(K_j)\subseteq T_0$ 
         for all $j\geq 0$, and $\bigcap_{j=0}^\infty K_j = K$;\label{item:eremenkohalfstrips}
       \item $f(\overline{D})\subseteq D$;\label{item:eremenkodisc}
       \item $f^{N_j+1}(\partial K_j)\subseteq D$;\label{item:eremenkoboundary}
       \item $f^{N_j+1}$ is injective on $K_{j+1}$, with $f^{N_j+1}(K_{j+1}) \subseteq V_{j+1}(f)$;\label{item:eremenkoimage}
       \item if $z = t + 7i/2$ with $1/t_j \leq t \leq t_j$, then $\lvert \re f^{N_j+1}(z)\rvert \leq 1$;\label{item:eremenkobungee}
       \item if $N_j+1\leq n \leq N_{j+1}$, then $\lvert \re f^n(\zeta)\rvert \geq j$.\label{item:eremenkoescaping}
     \end{enumerate}
     
  Observe that this proves the theorem. Indeed, let $z\in K\setminus\{\zeta\}$. Then by~\ref{item:eremenkoimage} and~\ref{item:eremenkobungee}, for $j\geq 0$, 
    $\lvert \re f^{N_j+1}(z)\rvert \leq 1$ and $1\leq \im f^{N_j+1}(z)\leq 2$, while 
    $f^{N_{j+1}}(z) \in T_{j+1}$, and hence $f^{N_{j+1}}(z)\to\infty$ as $j\to\infty$. So $z\in \BU(f)$. On the other hand, 
     $\zeta\in I(f)$ by~\ref{item:eremenkoescaping}. Furthermore, by~\ref{item:eremenkodisc} and ~\ref{item:eremenkoboundary}, $\partial K_j$ belongs to 
     $F(f)\setminus (I(f)\cup \BU(f))$. 
    So, by~\ref{item:eremenkohalfstrips}, $\bigcup_{j=0}^\infty \partial K_j$ separates $K$ from every other point of $J(f)\cup I(f)\cup \BU(f)$.
    \begin{figure}
\begin{center}
\def\svgwidth{\textwidth}
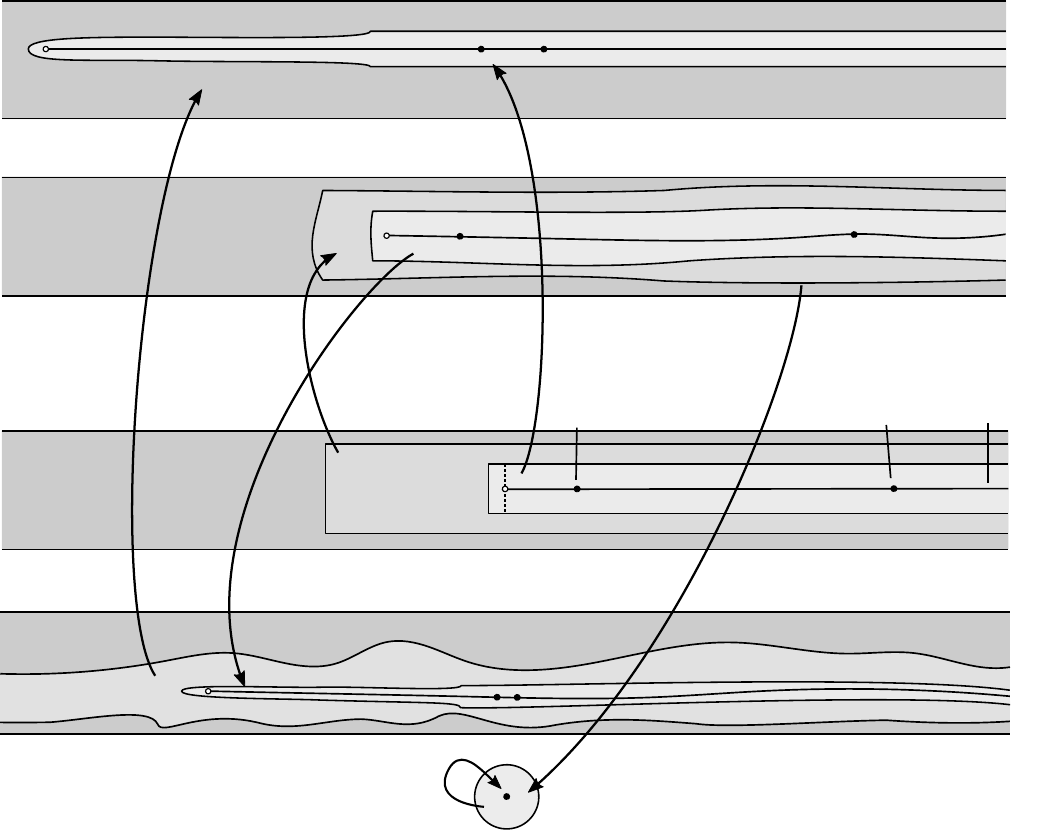 
\end{center}
\vspace*{-10pt}
\caption{The definition of $g_{j+1}$ in the proof of Theorem~\ref{thm:eremenkopoint}. The set $\tilde{L}_j \subseteq K_j \subseteq T_0$ is the light grey half-strip containing the ray $K$. On $\interior(f_j^{N_j}(\tilde{L}_j))$, the map $g_{j+1}$ is the composition of three conformal maps: 
  $(f_j^{N_j}|_{\interior(\tilde{L}_j)})^{-1}$,  $\tilde{\phi}_j$, and $\left(f_j^{j+1}|_{V_{j+1}(f_j)}\right)^{-1}$.}
    \label{fig:eremenkopointproof}
\end{figure}

  To construct the function $f$, we proceed similarly as in the proof of Proposition~\ref{prop:mainmaverick}, but will need to take some care because now the sets $K$ and $K_j$ are unbounded. Thus, we once more construct a 
    sequence of entire functions $f_j$, each of which 
    approximates a function $g_j$ (defined inductively in terms of $f_{j-1}$)
    on a set $A_j$ up to an error $\eps_j$, where $A_j$ satisfies the hypotheses
    of Arakelyan's theorem. Again define 
   \[ \Sigma_j \defeq \overline{S} \cup \bigl\{ z\in\C\colon \lvert 4\im z \rvert \leq 5^{j+1} + 3\bigr\}. \]
   Along with $f_j$, we construct a half-strip $K_j$ 
    as in~\ref{item:eremenkohalfstrips}, in such a way that the following
    inductive hypotheses are satisfied: 
     \begin{enumerate}[(i)]
       \item $f_j^{N_j}$ is injective on $K_j$, 
            $f_j^{N_j}(K_j)\subseteq T_j$ and $\dist(f_j^{N_j}(K_j),\partial T_j)>0$;\label{item:eremenkofj}
       \item $\re f_j^{N_j}(z) \to +\infty$ as $z\to\infty$ in $K_j$;\label{item:eremenkoinfinity}
        \item on the set $\widehat{K}_j \defeq \bigcup_{k=0}^{N_j-1} f_j^k(K_j)$, the derivative
                 $\lvert f_j'\rvert$ is bounded from above and below by positive constants, and $f_j$ is uniformly continuous at every point 
                 of $\widehat{K}_j$ in the sense of Remark~\ref{rmk:uniformcontinuity};\label{item:eremenkoderivative}
        \item $\Sigma_{j-1}\subseteq A_j \subseteq \Sigma_j$ if $j\geq 1$;\label{item:eremenkoAj}
        \item $\eps_j \leq \eps_{j-1}/2$ if $j\geq 1$.\label{item:eremenkoeps}
     \end{enumerate}
   
   We choose $K_0 \subseteq \interior(T_0)$ to be an arbitrary closed half-strip with $K\subseteq \interior(K_0)$  
     (see Figure~\ref{fig:eremenko}). 
    The inductive hypotheses hold trivially for $j=0$ (recall that $N_0=0$).
        
    Suppose that $f_j$ and $K_j$ have been constructed. 
             Let $L_j$ be a closed horizontal half-strip contained in $\interior(K_j)$ such that
             $\zeta\in\partial L_j$ and $K\setminus \{\zeta\}\subseteq \interior(L_j)$. 
           Let $\delta_j$ denote the hyperbolic length, in the hyperbolic metric of $\interior(L_j)$, of the segment $[\zeta+1/t_j,\zeta+t_j]$.
        
        By~\ref{item:eremenkofj} and~\ref{item:eremenkoinfinity}, 
         $Q_j\defeq f_j^{N_j}(\partial K_j)$ is a Jordan arc with both ends at infinity and $\dist(Q_j,\partial T_j)>0$. By~\ref{item:eremenkoderivative}, also
        \smash{$\dist(Q_j, f_j^{N_j}(L_j)) > 0$}. Let $R_j\subseteq T_j$ be a closed neighbourhood of $Q_j$ homeomorphic to a bi-infinite closed strip (with real part tending to $+\infty$ in both directions),
           such that $\dist(Q_j,\partial R_j)>0$ and \smash{$\dist(R_j , f_j^{N_j}(L_j))>0$}. 
       
     Now let $W_j \subseteq T_{j+1}$ be a bi-infinite horizontal strip. If $W_j$ is chosen sufficiently thin, then
       any hyperbolic ball of radius $\delta_j$ (in the hyperbolic metric of $W_j$) has Euclidean diameter less than $1$. 
       Let 
       $\phi_j \colon \interior(L_j) \to W_j$ be a conformal isomorphism with $\re \phi_j(\zeta+1) = 0$ whose continuous extension to the boundary satisfies
       $\phi_j(\zeta) = -\infty$ and $\phi_j(+\infty) = +\infty$. Then $\lvert \re\phi_j( \zeta +t )\rvert < 1$ for $1/t_j\leq t \leq t_j$. Observe that $\phi_j(z)$ is finite for all
       $z\in \partial L_j\setminus\{\zeta\}$. 
       
     Let $\eta>0$ be small (see below); define $\tilde{L}_j \defeq L_j - \eta/2$ and $\tilde{\phi}_j(z) \defeq \phi_j(z + \eta)$. Then $\tilde{\phi}_j\colon \tilde{L}_j\to\overline{W_j}$ is continuous. Set 
         \[ A_{j+1} \defeq \Sigma_j \cup R_j \cup f_j^{N_j}(\tilde{L}_j) \]
         and
            \begin{equation}\label{eqn:eremenkog} g_{j+1}\colon A_{j+1}\to \C; \quad
              z \mapsto \begin{cases} f_j(z), &\text{if }z\in\Sigma_j, \\
                   0, &\text{if }z\in R_j, \\
                   (f_j^{j+1}|_{V_{j+1}(f_j)})^{-1}\bigl(\tilde{\phi}_j( (f_j^{N_j}|_{\tilde{L}_j})^{-1}(z))\bigr), &\text{if }z\in f_j^{N_j}(\tilde{L}_j). \end{cases} \end{equation}
   (See Figure~\ref{fig:eremenkopointproof}.)  
     If $\eta$ is chosen sufficiently small, then
      \begin{itemize}
         \item $\tilde{L}_j\subseteq K_j$ and $f_j^{N_j}(\tilde{L}_j)$ is disjoint from $R_j$;
         \item $\lvert\re g_{j+1}^{N_j+1}(\zeta)\rvert = \lvert \re (( f_j^{j+1}|_{V_{j+1}(f_j)})^{-1}(\tilde{\phi}_j(\zeta)))\rvert > j$;
         \item $\lvert \re g_{j+1}^{N_{j+1}}(\zeta+ t)\rvert = \lvert \re \tilde{\phi}_j(\zeta+ t + \eta )\rvert < 1$ for $1/t_j \leq t \leq t_j$.
      \end{itemize}
            (For the equality in the final bullet point, recall that $N_{j+1} = N_j + j + 2$.)
    Let $K_{j+1}\supseteq K$ be any closed half-strip contained in the interior of $\tilde{L}_j$ and with $K\subseteq \interior(K_{j+1})$,  chosen sufficiently small such that 
      \begin{equation}\max_{z\in \partial K_{j+1}} \dist(z, K) \leq \frac{1}{j+1}.\end{equation} 

   The set $A_{j+1}$ is the union of three 
       sequences of pairwise disjoint topological closed strips and half-strips, each tending uniformly to $\infty$, 
       and therefore satisfies the hypotheses of Arakelyan's theorem. 
       
       Observe that  $g_{j+1}^{N_{j+1}}$ is defined and injective on $\tilde{L}_j$. We claim that $\lvert g_{j+1}'\rvert$ is bounded from above 
       and below by positive constants on $g_{j+1}^k(\tilde{L}_j)$ for $0\leq k < N_{j+1}$. 

      This holds by the inductive hypothesis for $0\leq k < N_j$. Since $g_{j+1}^{N_j+1}(\tilde{L}_j)\subseteq V_{j+1}(f_j)$, the claim
        also holds for $N_j+1 \leq k < N_{j+1}$, by Lemma~\ref{lem:strips}. So it remains to establish it for $k=N_j$. The derivatives of 
        the iterated inverse branches of $f_j$ used
       in the definition of $g_{j+1}$ on $f_j^{N_j}(\tilde{L}_j)$ are bounded from above and below by positive constants, again by the inductive hypothesis and
       Lemma~\ref{lem:strips}. For $\lvert \tilde{\phi}_j'\rvert$, this follows from
        Lemma~\ref{lem:stripderivative}, applied to a map obtained from $\tilde{\phi}_j$ by pre- and post-composition with affine maps.

  In particular, $\dist(g_{j+1}^k(K_{j+1}),g_{j+1}^k(\partial\tilde{L}_j))>0$ for $0\leq k\leq N_j$, and 
     $g_{j+1}$ is uniformly continuous at every point of $g_{j+1}^k(K_{j+1})$ for $0\leq k < N_j$. Observe furthermore that $g_{j+1}$ is uniformly continuous at every point of
     $g_{j+1}^k(K_j)$ for $0\leq k < N_j$ 
     by~\ref{item:eremenkoderivative}, and also, trivially, at every point of $g_{j+1}^{N_j}(\partial K_j)=Q_j\subseteq R_j$.
               
  We now claim that, if $\eps_{j+1}\leq \eps_j/2$ is sufficiently small, then any entire function $f$
      with $\lvert f(z) - g_{j+1}(z)\rvert \leq 2\eps_{j+1}$ on $A_{j+1}$ satisfies the following properties.
      \begin{enumerate}[(1)]
         \item $f^{N_{j+1}}$ is injective on $K_{j+1}$, with $\lvert f'\rvert$ bounded above and below by positive constants on
            $\bigcup_{k=0}^{N_{j+1}-1}f^k(K_{j+1})$, and $f$ is uniformly continuous at every point of this set;\label{item:e_ind_deriv}
         \item $f^{N_j+1}(\partial K_j) \subseteq D$;\label{item:e_ind_boundary}
         \item $f^{N_j+1}(K_{j+1})\subseteq V_{j+1}(f)$ and $\dist(f^{N_{j+1}}(K_{j+1}),\partial T_{j+1})>0$;\label{item:e_ind_Vj}
         \item $\lvert \re f^{N_j+1}(\zeta)\rvert > j$;\label{item:eremenkoescaping2}
         \item $\lvert \re f^{N_{j+1}}(\zeta+ t)\rvert < 1$ for $1/t_j \leq t \leq t_j$.\label{item:eremenkobungee2}
      \end{enumerate}

         Observe that all of these properties hold for $g_{j+1}$ itself. Then~\ref{item:e_ind_boundary} to~\ref{item:eremenkobungee2} follow
            from Lemma~\ref{lem:approx2},  which we may apply because of the
            above facts concerning the uniform continuity of $g_{j+1}$. Property~\ref{item:e_ind_deriv} follows from Corollary~\ref{cor:approx},
            applied with $G = \bigcup_{k=0}^{N_{j+1}}(\interior(\tilde{L}_j))$.
     
            Apply Arakelyan's theorem to obtain a function $f_{j+1}$, with
         $\lvert f_{j+1}(z) - g_{j+1}(z)\rvert \leq \eps_{j+1}$ on $A_{j+1}$. 
         Then~\ref{item:eremenkofj} holds by~\ref{item:e_ind_deriv} and~\ref{item:e_ind_Vj}. Property~\ref{item:eremenkoinfinity} holds
         by definition of $g_{j+1}$. Property~\ref{item:eremenkoderivative} also holds by~\ref{item:e_ind_deriv}, while~\ref{item:eremenkoAj}
          and~\ref{item:eremenkoeps} are immediate from the definitions of $A_{j+1}$ and $\eps_{j+1}$. This
          concludes
         the inductive construction. 
         
         By~\ref{item:eremenkoAj}, we have 
         $A_{j+1}\supseteq A_j$ and $\bigcup_{j=1}^\infty A_j = \C$. 
         By~\ref{item:eremenkoeps}, the functions $(f_j)$ form a Cauchy sequence
         on every $A_j$, and thus converge locally uniformly to 
         an entire function~$f$ with $\lvert f(z) - g_j(z)\rvert \leq 2\eps_j$ for all $j\geq 0$
         and all $z\in A_j$. 
         
        Properties~\ref{item:eremenkohalfstrips}  to~\ref{item:eremenkoimage} follow directly from the construction. 
       Indeed, 
        Property~\ref{item:eremenkohalfstrips} holds by
         definition of $K_j$ and property~\ref{item:eremenkodisc} by
         choice of $\eps_0$ and Lemma~\ref{lem:strips}.
         Claim~\ref{item:eremenkoboundary} is a consequence of~\ref{item:e_ind_boundary}, while~\ref{item:eremenkoimage}
            follows from~\ref{item:e_ind_deriv} and~\ref{item:e_ind_Vj}. 

       Finally, property~\ref{item:eremenkobungee} follows 
         from~\ref{item:eremenkobungee2} by the
         statement on real parts in 
         Lemma~\ref{lem:strips}; likewise,~\ref{item:eremenkoescaping} follows from~\ref{item:eremenkoescaping2}. 
\end{proof} 

\begin{rmk}
Every point of $J(f)$ is an accumulation point of unbounded connected components of $I(f)$ \cite[Theorem~1]{rippon-stallard05}. Using the notation of the proof of Theorem~\ref{thm:eremenkopoint}, it follows that $X=\{\zeta\}$ is accumulated on by unbounded components of $I(f)$ lying in strips of the form \mbox{$\textup{int}(K_{n})\setminus K_{n+1}$}.
\end{rmk}

\begin{rmk}\label{rmk:unbounded-wd}
We can modify the proof of Theorem~\ref{thm:eremenkopoint} by letting the height of the half-strips $(K_n)$ tend to a positive constant rather than to $0$, so that $K=\bigcap K_n$ is a closed half-strip. In this way, we obtain a transcendental entire function $f$ for which $U=\textup{int}(K)$ is an (unbounded) oscillating wandering domain such that $\partial U\cap I(f)\neq\emptyset$. The set $\partial U \cap I(f)$ has zero harmonic measure in $\partial U$ by~\cite[Theorem~1.2(a)]{rippon-stallard11}, and is therefore
 totally disconnected. Hence $\partial U$ contains a singleton component of $I(f)$.
\end{rmk}

We now prove Theorem~\ref{thm:eremenkocounterexample}, which is a more general version of Theorem~\ref{thm:eremenkopoint}.

\begin{proof}[Proof of Theorem~\ref{thm:eremenkocounterexample}]
Let 
    $X\subseteq\C$ be a full plane continuum. As in the proof of Theorem~\ref{thm:eremenkopoint}, we may assume that $X\subseteq T_0$. 
    Let $\zeta\in X$ be a point with maximal real part, and let $K$ be the union of 
    $X$ with the horizontal ray $\zeta + [0,+\infty)$. 
    
  We proceed with the recursive construction as before, but must adjust the
    definition of $\tilde{L}_j$ and $K_{j+1}$. These sets can no longer be chosen
    to be straight horizontal strips. Instead, each such set will be what we will term a \emph{decorated horizontal half-strip}, that is, 
    a topological half-strip whose intersection with some right half-plane is a straight horizontal half-strip.
    Let $K_0$ be 
    any decorated horizontal half-strip such that $K_0\subseteq T_0$ and
    $K\subseteq \interior(K_0)$.
    
    In the recursive step of the construction (when $K_j$ has been defined), we first let 
     $L_j\subseteq \interior(K_j)$
    be a closed horizontal half-strip with
    $\zeta+\eps \in\partial L_j$ and $\zeta + (\eps,+\infty)\subseteq \interior(L_j)$,
    where $\eps < 1/t_j$. 
    Exactly as in the proof of Theorem~\ref{thm:eremenkopoint}, we find a 
    conformal map $\phi_j\colon \interior(L_j)\to W_j$, where $W_j\subseteq T_{j+1}$ 
     is a horizontal strip such that $\phi_j(+\infty) = +\infty$, 
     $\phi_j(\zeta + \eps) = -\infty$, and $\lvert \re \phi_j(\zeta + t)\rvert < 1$ for
        $1/t_j \leq t \leq t_j$. 
        
   Now let $\hat{L}_j\supseteq L_j$ be obtained from 
     $L_j$ by adding a small neighbourhood of $X$ to $L_j$, chosen in such a way that $\hat{L}_j$ is a decorated horizontal half-strip. (That this is possible
     follows from Lemma~\ref{lem:Kn}.) Let 
     $\hat{\phi}_j\colon \interior(\hat{L}_j)\to W_j$ be a conformal isomorphism with 
     $\hat{\phi}_j(\zeta+1) = \phi_j(\zeta+1)$ and $\hat{\phi}_j(+\infty)=+\infty$. We suppose that 
     $\interior(\hat{L}_j)$ is sufficiently close to $\interior(L_j)$ in the sense
     of Carath\'eodory kernel convergence (seen from $\zeta + 1$); compare \cite[Section~1.4]{pommerenke92}. Then $\hat{\phi}_j$ still maps 
     $\zeta + [1/t_j,t_j]$ to points at real parts less than~$1$, while 
         \begin{equation}\label{eqn:eremenkoXbig} \big\lvert \textup{Re}\, (f_j^{j+1}|_{V_{j+1}(f_j)})^{-1}(\hat{\phi}_j(z))\big\rvert > j \end{equation}
       for all $z\in X$. 
       
     We now let $K_{j+1} \subseteq \tilde{L}_j \subseteq \hat{L}_j$ be 
       slightly smaller topological half-strips and let 
       $\tilde{\phi}_{j}$ be the restriction of $\hat{\phi}_j$ to $\tilde{L}_j$. We define 
       $g_{j+1}$ as in~\eqref{eqn:eremenkog}. Just as in the proof of Theorem~\ref{thm:eremenkopoint}, the derivative 
       $\lvert \tilde{\phi}_{j}'\rvert$ is bounded from above and below on $\tilde{L}_j$ by Lemma~\ref{lem:stripderivative}. The 
       remainder of the proof proceeds as before, except that~\ref{item:eremenkoescaping2} is replaced by the
       condition that $\lvert \re f^{N_j+1}(z)\rvert > j$ for all $z\in X$, which is possible by~\eqref{eqn:eremenkoXbig}.
       
   For the resulting function $f$, it follows that $X\subseteq I(f)$, while $\zeta + (0,\infty)\subseteq \BU(f)$. Furthermore, 
    every point of $K$ is separated from any other point of $\C$ by some element of $\partial K_j$. As all points of $\partial K_j$ are eventually mapped to
    $D$, we see that $K$ is a connected component of $I(f)\cup \BU(f)$, and $X$ is a connected component of $I(f)$, as claimed.
\end{proof}

\begin{rmk}\label{rem:order-classB}
For the function $f$ constructed in the proof of Theorem~\ref{thm:eremenkocounterexample}, the set \mbox{$f^{-1}(\C\setminus \overline{D})$} has infinitely many components (at least one inside each strip $T_n$). Hence $f$~is of infinite order by the Denjoy--Carleman--Ahlfors theorem. On the other hand, the fact that $|f'|$ is bounded above on an unbounded set for which $|f|$ is large contradicts the expansivity property of functions in the class $\mathcal B$ \cite[Lemma~1]{eremenko-lyubich92} (see also \cite[Theorem~1.1]{rempe21}), so $f\notin \mathcal B$.     
\end{rmk}

To conclude the section, we show that~-- as mentioned in the introduction~--
 the condition that $K$ is full 
  in Theorem~\ref{thm:eremenkocounterexample} is necessary, but the 
  condition that $K$ is compact is not.
\begin{prop}\label{prop:compact-comp}
 Let $f$ be a transcendental entire function, and suppose that $K\subseteq I(f)$ is a compact connected component
    of $I(f)$. Then $K$ is full. 
 
 On the other hand, there exists a transcendental entire function for which 
   $I(f)$ has a bounded connected component that is not compact. 
\end{prop}
\begin{proof}
 Suppose that $K\subseteq I(f)$ is compact and connected, and let 
   $U$ be a bounded connected component 
  of $\C\setminus K$. There are two possibilities:
  either $U$ is a component of the Fatou set of $f$, or $U$ intersects the Julia set. 
   In the first case, $U\subseteq I(f)$ since all boundary points of $U$ escape \cite[Theorem~1.2 (a)]{rippon-stallard11}. In the second, $U$ intersects 
   the fast escaping set $A(f)$ defined by~\eqref{eqn:fastescapingset}~\cite[Lemma~3]{bergweiler-hinkkanen99}. 
   Every connected component of $A(f)$ is unbounded~\cite[Theorem~1]{rippon-stallard05}, and hence
   intersects $K$. We conclude that, in either case, $K$ is not a connected component of $I(f)$, as claimed. 
   
 To prove the second part of the proposition,  
    we can apply the proof of Theorem~\ref{thm:eremenkocounterexample}
    with $K = \overline{\DD}$, but modified similarly as in     
    Section~\ref{sec:maverickconstruction}
    so that a countable subset $Y\subseteq \partial\DD = \partial K$ belongs to
    $\BU(f)$. In other words, $\DD\subseteq I(f)$, but $Y\cap I(f)=\emptyset$. So the connected component 
    of $I(f)$ containing $\DD$ is a proper subset of $\overline{\DD}$. (We omit
    the details.) 
\end{proof}

\section{Maverick points and harmonic measure}

\label{sec:maverick}

In this section, we prove Theorem~\ref{thm:maverick}, showing that the set of maverick points of a wandering domain has zero harmonic measure with respect to the wandering domain. We begin with the following
 observation. 
 \begin{lem}[Maverick points and spherical diameter]\label{lem:maverick}
 Let $f$ be a transcendental entire function and suppose that $U$ is a wandering domain of $f$. Let $\zeta\in U$ and $z\in \partial U$. Then $z$ is maverick if and only if
  $\dist^{\#}(f^n(\zeta),f^n(z))\not\to 0$ as $n\to\infty$. In particular, if $\diam^{\#}(f^n(U))\to 0$ as $n\to\infty$, then $U$ has no maverick points.
 \end{lem}
 \begin{proof}
   The first claim is a simple restatement of the definition of a maverick point. Indeed, if 
     $f^{n_k}(z)\to w\in\Ch$ and $f^{n_k}(\zeta)\not\to w$ as $k\to\infty$, then 
     $\dist^{\#}(f^{n_k}(z),f^{n_k}(\zeta))\not\to 0$ as $k\to\infty$. Conversely, if 
     $(n_k)_{k=0}^\infty$ is a subsequence such that $\dist^{\#}(f^{n_k}(z),f^{n_k}(\zeta)) \geq \eps>0$ for all $k\in\N$, then we may pass to a further subsequence $({n_{k_j}})_{j=0}^\infty$ such
     that $f^{n_{k_j}}(z)$ and $f^{n_{k_j}}(\zeta)$ converge to distinct points on the sphere.
     
     The second claim follows immediately from the first. 
 \end{proof}

Our proof of Theorem~\ref{thm:maverick} is similar to that of \cite[Theorem~1.5]{osborne-sixsmith16}, but somewhat simpler, despite giving a stronger result. 
  As with~\cite[Theorem~1.5]{osborne-sixsmith16}, 
  we use the following
   lemma, which is \cite[Lemma~4.1]{osborne-sixsmith16} and is related to     
    \cite[Theorem~1.1]{rippon-stallard11}.
\begin{lem}\label{lem:Harmonic-measure-zero}
Let $(G_k)_{k=0}^\infty$ be a sequence of pairwise disjoint simply connected domains in~$\C$. Suppose that, for each $k\in\N$, $g_k:\overline{G}_{k-1}\to\overline{G}_k$ is analytic in $G_{k-1}$, continuous in $\overline{G}_{k-1}$, and satisfies $g_k(\partial G_{k-1})\subseteq \partial G_k$. Let 
$$
h_k=g_k\circ \cdots \circ g_2\circ g_1\quad \textup{ for } k\in\N.
$$
Let $\xi\in \Ch$, $\delta \in (0,1)$, $c>1$, and $z_0\in G_0$. Then
\begin{align*}
H=\big\{z\in\partial G_0\colon &\dist^{\#}(h_k(z),\xi)\geqslant c\delta \textup{ and } 
  \dist^{\#}(h_k(z_0),\xi)<\delta \textup{ for infinitely many } k\big\}
\end{align*}
has harmonic measure zero relative to $G_0$.
\end{lem}

\begin{proof}[Proof of Theorem~\ref{thm:maverick}]
Let $f$ be a transcendental entire function with a wandering domain $U$. As
 mentioned in the introduction, if $U$ is multiply connected, then 
 $\overline{U} \subseteq I(f)$; see \cite[Theorem 2]{rippon-stallard05}. Hence, we assume that
  $U$ is simply connected and fix some $z_0 \in U$.

For $0<\delta<1$ and $\xi \in \hat{\mathbb{C}}$ define
$H(\xi,\delta)$ to consist of all points $z\in \partial U$ for which there are infinitely many $n\in\N$ such that 
$\dist^{\#}(f^{n}(z_0),\xi)<\delta$ and $\dist^{\#}(f^{n}(z),\xi)>2\delta$. We claim that the harmonic measure of $H(\xi,\delta)$ relative to $U$ is zero.

Let $(n_k)$ be the sequence of all times for which $\dist^{\#}(f^{n_k}(z_0),\xi)<\delta$. (If such a sequence does not exist, then $H(\xi, \delta)=\emptyset$ and there is nothing to prove.)

Set $G_0\defeq U$, $g_k\defeq f^{n_k-n_{k-1}}$, and let $G_k$ be the Fatou component of $f$ that contains the image $f^{n_k}(U)$. Then, each $G_k$ is a simply connected wandering domain and satisfies $g_k(\partial G_{k-1}) \subseteq \partial G_k$.  By Lemma~\ref{lem:Harmonic-measure-zero}, the harmonic measure of $H(\xi,\delta)$ relative to $U$ is zero as claimed. 

We claim that the union of the sets $H(\xi,\delta)$, where $\delta\in\Q$ and 
  $\xi\in \Q(i)$ are rational, 
 contains all maverick points. By Lemma \ref{lem:maverick}, if $z\in \partial U$ is a maverick point, then there exists $\xi'\in \omega(U,f)$ and an increasing sequence $(n_k)_{k=0}^\infty$ such that 
$f^{n_k}(z_0)\to \xi'$ but $f^{n_k}(z)\not\to \xi'$. Passing to a further subsequence if necessary, we may assume that there exists $\delta \in \mathbb{Q}$ with $0<\delta<1$ such that $\dist^{\#}(f^{n_{k}}(z_0),\xi')<\delta/2$ and $\dist^{\#}(f^{n_{k}}(z),\xi')>3\delta$
for all $k$. Hence, if $\xi\in\mathbb{Q}(i)$ with $\dist^{\#}(\xi,\xi')<\delta/2$, then $z\in H(\xi,\delta)$.
Thus, 
\[\big\{z\in\partial U: \text{$z$ is a maverick point}\big\} \subseteq \bigcup_{\xi\in \mathbb{Q}(i)} \bigcup_{\delta \in \mathbb{Q}\cap{(0,1)}} H(\xi,\delta).\]

Since $H(\xi,\delta)$ has zero harmonic measure relative to $U$ and by countable additivity of the harmonic measure, the set of maverick points has zero harmonic measure relative to~$U$.
\end{proof}
\begin{rmk}
  Note that the proof does not require us to distinguish whether $U$ is escaping,
     oscillating or orbitally bounded. (Recall from the introduction that the latter means that every point in $U$ has bounded
     orbit, and that it is unknown whether orbitally bounded wandering domains exist.)
\end{rmk}

\section{Proof of Proposition~\ref{prop:spikes}}\label{sec:spikes}

Proposition~\ref{prop:spikes} is a consequence of the following fact. 
\begin{prop}[Conformal maps taking boundary points to infinity]\label{prop:spikes2}
  Let $\Xi\subseteq \partial \DD$ be a compact set of zero logarithmic capacity. Then there exists a simply connected domain 
    $\Omega\subseteq \Sigma \defeq \{ x+iy\colon x > 0 \text{ and } \lvert y\rvert <\pi/2 \}$ and a conformal isomorphism
    $\phi\colon \DD\to \Omega$ such that $\phi$ extends continuously to $\partial\DD$ with
    $\phi(\xi)=\infty$ for all $\xi\in \Xi$.  
\end{prop}
\begin{proof}[Proof of Proposition~\ref{prop:spikes}, using Proposition~\ref{prop:spikes2}]
   First observe that we may assume that $V=\tilde{\Sigma}\defeq \{x + iy\colon \lvert y\rvert < \pi\}$, that 
     $\Delta$ contains all points of $V$ of modulus greater than some $R>0$, and that $\nu = 0$. 
      Indeed, otherwise let $\tilde{V}\subseteq V$ be a Jordan domain that contains $\nu$ and whose boundary passes through $\Delta$, and map 
      $\tilde{V}$ conformally onto $\tilde{\Sigma}$, taking $\nu$ to zero and some point of $\partial \tilde{V}\cap \Delta$ to $+\infty$. 
      
  Now let $\phi$ and $\Omega$ be as in Proposition~\ref{prop:spikes2}. For $0<\lambda<1$, define 
    \[ \tilde{\phi}\colon \DD\to \tilde{\Sigma}; \qquad z\mapsto \phi(\lambda z)-\phi(0). \]
    Since $\lvert \im \phi(0)\rvert < \pi/2$, this map does indeed take values in $\tilde{\Sigma}$, and clearly $\tilde{\phi}(0)=0$. 
    Moreover, if $\lambda$ is chosen sufficiently close to $1$, then $\re \tilde{\phi}(\xi)>R$ for all $\xi\in \Xi$. Setting
    $W \defeq \tilde{\phi}(\DD)$ for such $\lambda$, the proof is complete. 
\end{proof}    

Proposition~\ref{prop:spikes2} follows from a powerful result
 of Bishop about boundary interpolation sets for conformal maps~\cite{bishop2006}.

\begin{proof}[Proof of Proposition~\ref{prop:spikes2} using {\cite[Theorem~1]{bishop2006}}]
Fix an interval $J\subsetneq\partial \DD$ that contains
  $\Xi$. We may choose a homeomorphism 
   $g\colon\DD\to \Sigma$ such that $g$ extends continuously to $\partial\DD$ with 
   $g(\xi)=\infty$ for all $\xi\in J$. Of course, $g$ cannot be conformal since it collapses a whole interval of the unit circle to a point. 
   However, by~\cite[Theorem~1]{bishop2006}, there is
   a domain $\Omega\subseteq\Sigma$ and a conformal isomorphism $f\colon \DD\to \Omega$ that extends continuously to $\partial\DD$
   and $f|_{\Xi} = g|_{\Xi}$. 
 \end{proof}

  In fact, Proposition~\ref{prop:spikes2} is considerably weaker 
       than~\cite[Theorem~1]{bishop2006}. Indeed, it is a preliminary step in the
   proof of~\cite[Theorem~6]{bishop2006}, which in turn is used in the proof of~\cite[Theorem~1]{bishop2006}. 
    For the reader's convenience, and in order to be able to discuss how the proof simplifies when
    $\Xi$ is finite, we sketch a direct proof of Proposition~\ref{prop:spikes2},
    following~\cite{bishop2006}.

\begin{proof}[Proof of Proposition~\ref{prop:spikes2}]
   Let us move from the unit disc $\DD$ to the upper half-plane $\HH^+$. We are given a compact set $E\subseteq \R$ of zero logarithmic capacity,
     and wish to construct a conformal isomorphism between $\HH^+$ and  a subset of $\Sigma$ that maps all points of $E$ to $\infty$. 
     
  When $E$ is finite, such a conformal isomorphism can be explicitly written down. Let $p_1,\dots,p_n$ be the elements of $E$, in increasing order. 
     Set $\Sigma' = \{ x + iy\colon 0 < y < \pi\}$ and define 
        \begin{equation}\label{eqn:psi} \psi\colon \HH^+\to\Sigma'; \quad z\mapsto \pi i 
        - \frac{1}{n} \cdot \sum_{j=1}^n \Log(z-p_j),
        \end{equation}
   where $\Log$ is the principal branch of the logarithm. Since $0<\arg(z-p_j)<\pi$ for $z\in \HH^+$ and all $j$, we see that $\psi$ does indeed take values in $\Sigma'$.
    Moreover, 
      \[ \psi'(z) = \frac{-1}{n}\cdot \sum_{j=1}^{n} \frac{1}{z-p_j}. \]
   Since $\im(z-p_j)=\im z>0$ for all $z\in\HH^+$ and all $j$, it follows that $\im \psi'(z)>0$ for all $z\in\HH^+$. Hence 
     $\psi$ is a conformal map defined on $\HH^+$ (see~\cite[Lemma~1]{eremenkoyuditskii12}). 
     Clearly $\re \psi(z)\to +\infty$ as $z\to p_j$ for any $j$, and $\psi$ extends continuously to $\R\cup\{\infty\}$ (with the convention that it takes
     the value $+\infty$ 
     for all $z\in  E$, and the value $-\infty$ at $\infty$). 
    We obtain the desired map by postcomposing $\psi$ with a conformal isomorphism from $\Sigma'$ to $\Sigma$ that fixes $+\infty$, and the
     proof (in the case of finite $E$) is complete.
     
    For the more general case, where $E$ has zero logarithmic capacity, we replace the sum~\eqref{eqn:psi} by an integral, using 
     a theorem of Evans; see \cite[Theorem~E.2]{garnett-marshall05}. This states that there exists a probability measure $\mu$ supported on $E$ such that the potential
       \[ u(z) = \int_{\R} -\Log\lvert z-x\rvert \,\deriv \mu(x) \]
       tends to infinity at every point of $E$. Now define 
        \begin{equation}\label{eqn:newpsi} \psi\colon \HH^+\to\Sigma'; \quad z\mapsto \pi i - \int_{\R} \Log(z-x)\,\deriv\mu(x).
        \end{equation}
       Then $u = \re \psi$, and in particular $\psi(z)\to +\infty$ as $z\to E$. The proof now proceeds exactly as in the case of finite $E$. 
\end{proof}
\begin{rmk}
  Observe that the formula~\eqref{eqn:psi} for the finite case is a special case of~\eqref{eqn:newpsi}. Conversely, the function $\psi$ in~\eqref{eqn:newpsi} 
    is the limit of the corresponding functions~\eqref{eqn:psi} for a suitably chosen sequence of finite subsets of $E$. (See~\cite[Proof of Theorem~E.2]{garnett-marshall05}.) 
    That this can be done in such a way that $\psi(z)\to\infty$ at all points of $E$ follows from (and is indeed equivalent to) the fact that $E$ has
    logarithmic capacity zero; see~\cite[Theorem~E.1]{garnett-marshall05}. 
    
    \begin{figure}
\begin{center}
\def\svgwidth{\textwidth}
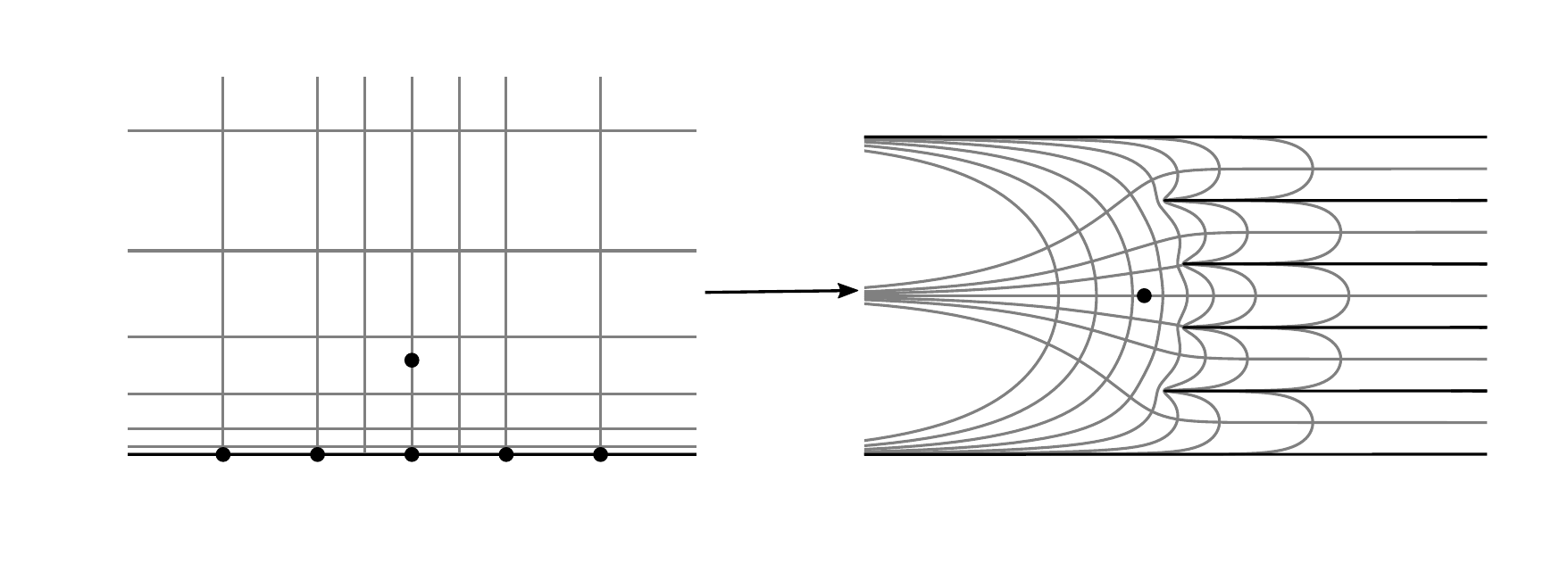 
\end{center}
\caption{The conformal map $\psi: \HH^+ \to \Sigma'$ in the proof of Proposition~\ref{prop:spikes2} with the points $p_j\in E$ such that $\psi(p_j)=\infty$.}\label{fig:strip}
\end{figure}

  The formula~\eqref{eqn:psi} is a Schwarz-Christoffel mapping for a degenerate polygon 
     with $\#E + 1$ zero angles (at the images of $-\infty$ and the points of $E$) and $\#E-1$ angles of $2\pi$ (at the images of the critical points of $\psi$, of which there 
     is one in every bounded complementary interval of $E$). For further discussion of these maps, and their relationship to real polynomials
     with only real critical points, see~\cite[Section~2]{eremenkoyuditskii12}.
\end{rmk}

\newcommand{\etalchar}[1]{$^{#1}$}
\providecommand{\bysame}{\leavevmode\hbox to3em{\hrulefill}\thinspace}
\providecommand{\MR}{\relax\ifhmode\unskip\space\fi MR }
\providecommand{\MRhref}[2]{%
  \href{http://www.ams.org/mathscinet-getitem?mr=#1}{#2}
}
\providecommand{\href}[2]{#2}

\end{document}